\documentclass{amsart}
\usepackage{amsthm,stmaryrd,amsmath,amscd}
\usepackage{amssymb}
\usepackage{epsfig}
\usepackage[usenames,dvipsnames]{color}
\usepackage{verbatim}
\usepackage{hyperref,subfig}
\usepackage{dsfont}
\usepackage{mathrsfs,float}
\usepackage[all]{xy}
\usepackage[active]{srcltx}
\usepackage[utf8]{inputenc}
\newcommand{\Proj}{{\mathsf {Proj}}}

\newcommand{\coev}{\stackrel{\longrightarrow}{\operatorname{coev}}}
\newcommand{\ev}{\stackrel{\longrightarrow}{\operatorname{ev}}}
\newcommand{\tev}{\stackrel{\longleftarrow}{\operatorname{ev}}}
\newcommand{\tcoev}{\stackrel{\longleftarrow}{\operatorname{coev}}}
\newcommand{\brk}[1]{{\left\langle{#1}\right\rangle}}

\newcommand{\ve}{\varepsilon}

\newcommand{\ro}{r}

\newcommand{\PP}{{\mathsf P}}

\newcommand{\Hr}{H_r}
\newcommand{\qr}{{q}}
\newcommand{\coh}{\omega}

\newcommand{\e}{{\operatorname{e}}}
\newcommand{\slt}{{\mathfrak{sl}(2)}}
\newcommand{\Uq}{{U_q\slt}}
\newcommand{\UqMed}{{\wb U_q\slt}}
\newcommand{\UqSm}{{\widetilde{U}_q\slt}}
\newcommand{\UsltH}{{U_q^{H}\slt}}
\newcommand{\Ubar}{{\wb U_q^{H}\slt}}
\newcommand{\unit}{\ensuremath{\mathbb{I}}}
\newcommand{\cat}{\mathscr{C}}

\newcommand{\Id}{\operatorname{Id}}
\newcommand{\bp}[1]{{\left(#1\right)}}

\newcommand{\qn}[1]{{\left\{#1\right\}}}
\newcommand{\qN}[1]{{\left[#1\right]}}

\newcommand{\qd}{{\mathsf d}}
\newcommand{\qdim}{\operatorname{qdim}}
\newcommand{\End}{\operatorname{End}}
\newcommand{\Hom}{\operatorname{Hom}}

\newcommand{\Span}{\operatorname{Span}}
\newcommand{\tr}{\operatorname{tr}}
\newcommand{\ptr}{\operatorname{ptr}}
\newcommand{\C}{\ensuremath{\mathbb{C}} }
\newcommand{\Z}{\ensuremath{\mathbb{Z}} }

\newcommand{\N}{\ensuremath{\mathbb{N}} }

\newcommand{\wb}{\overline}

\newcommand{\ms}[1]{\mbox{\tiny$#1$}}

\newcommand{\Cp}{{\ddot\C}}

\newcommand{\et}{{\quad\text{and}\quad}}

\renewcommand{\vec}{\mathsf w}
\newcommand{\vecl}{\vec^L}
\newcommand{\vecr}{\vec^R}
\newcommand{\vech}{\vec^H}
\newcommand{\vecs}{\vec^S}
\newcommand{\Cas}{C}
\newcommand{\cas}{c}
\newcommand{\ideal}{\mathcal{I}}
\newcommand{\md}{\operatorname{\mathsf{d}}}
\newcommand{\mt}{\operatorname{\mathsf{t}}}
\newcommand{\qt}{\mathsf t} 
\newcommand{\ob}{\operatorname{Ob}(\cat)}
 
\newcommand{\Tc}{\mathcal T} 

\newcommand{\epsh}[2]
         {\begin{array}{c} \hspace{-1.3mm}
        \raisebox{-4pt}{\epsfig{figure=#1,height=#2}}
        \hspace{-1.9mm}\end{array}}

\newtheorem{theo}{Theorem}[section]
\newtheorem{lemma}[theo]{Lemma}
\newtheorem{prop}[theo]{Proposition}
\newtheorem{cor}[theo]{Corollary}

\theoremstyle{definition}
\newtheorem{defi}[theo]{Definition}
\newtheorem{rem}[theo]{Remark}

\newtheorem{exo}[theo]{Exercise}

\theoremstyle{remark}

\newcounter{exo} \newcounter{numexercice}
\renewcommand{\theexo}{\arabic{exo}}

\begin{document}
\title[Some remarks on the unrolled quantum group of $\slt$]{Some remarks on the unrolled quantum group of $\slt$}

\author[F. Costantino]{Francesco Costantino}
\address{Institut de Math\'ematiques Toulouse\\
  118 route de Narbonne\\
 31062 Toulouse Cedex 9, France}
\email{Francesco.Costantino@math.univ-toulouse.fr}

\author[N. Geer]{Nathan Geer}
\address{Mathematics \& Statistics\\
  Utah State University \\
  Logan, Utah 84322, USA}
  \email{nathan.geer@gmail.com}

\author[B. Patureau-Mirand]{Bertrand Patureau-Mirand}
\address{UMR 6205, LMBA, universit\'e de Bretagne-Sud, universit\'e
  europ\'eenne de Bretagne, BP 573, 56017 Vannes, France }
\email{bertrand.patureau@univ-ubs.fr}
\thanks{The first author's research was supported by 
French ANR project ANR-08-JCJC-0114-01. Research of the second author was  partially supported by 
NSF grants  DMS-1007197 and DMS-1308196.   All the authors would like to thank  the Erwin Schr\"odinger Institute for Mathematical Physics in Vienna for support during a stay in the Spring of 2014, where part of this work was done.
 }\

\begin{abstract}
  In this paper we consider the representation theory of a
  non-standard quantization of $\slt$.  This paper contains several
  results which have applications in quantum topology, 
  including the classification
  of projective indecomposable modules 
  and a description of morphisms between  
  them.  In the process of proving these results the paper acts as a
  survey of the known representation theory associated to this
  non-standard quantization of $\slt$.  The results of this paper are used extensively in \cite{BCGP} to study Topological Quantum Field Theory (TQFT) and have connections with Conformal Field Theory (CFT).  
 \end{abstract}

\maketitle
\setcounter{tocdepth}{3}


\section{Introduction}
There are many different flavors of quantum $\slt$ based on a common
algebraic presentations.  In particular, these presentations depend on two  features: (1) if the quantum parameter $q$ is generic or a root of unity and (2) what
part of the center is killed.  The associated representation theory varies widely when these features are changed.  Two examples, when $q$ is a root of unity, are the finite-dimensional Hopf algebra commonly known as the small quantum group and the non-restricted quantum group obtained by specializing the De Concini-Kac form (for definitions of these algebras see \cite{CP}).   
The representation theory of the small quantum group leads to a modular category (in particular a finite, semi-simple, ribbon category) which can be used to construct 3-manifold invariants.  On the other hand, the representation theory of the non-restricted quantum group contains an infinite class of modules called the cyclic modules.  

In this paper we consider an intermediate quotient $\Ubar$, which we call the unrolled quantum group,   
leading to a category $\Ubar$-mod which is ribbon but not semi-simple or finite.  
This category has been used to construct quantum link and 3-manifold invariants in several papers \cite{ADO,GPT,GPT2,GP1,CM,CGP1,BCGP}.  These 3-manifold invariants have powerful new properties, including asymptotic behavior related to the Volume Conjecture and novel quantum representation of mapping class groups (see \cite{CGP1,BCGP}).  The existence of these properties is directly related to the unique representation theory discussed in this paper.

The Hopf algebra $\Ubar$ of this paper has an additional generator $H$
which is not in the usual quantum algebra associated to $\slt$.  The
element $H$ should be thought of as a logarithm of the usual generator
$K$.  The generator $H$ is used   
to define a braiding and a twist on a
category of $\Ubar$-modules.  In this category, $H$ is also
responsible for the apparition of an infinite cyclic group of
one dimensional invertible objects which  play a key role in the topological
applications.

The purpose of this paper is to give a survey of the known results
about $\Ubar$-mod while proving some new useful results which have
topological applications.  In particular, we classify all
indecomposable projective modules and define a modified trace on these
objects (see Section \ref{S:projectives}).  We give a ``graded''
quiver which describes the maps between the indecomposable projective
modules (see Section \ref{S:AlgProj}).  We also study the
decomposition of the tensor product of certain indecomposable modules
(see Section \ref{S:DecOfTensorPro}).  These results are used in
\cite{BCGP} in an essential way to build a TQFT for 3-manifolds
equipped with a cohomology class.  The category $\Ubar$-mod contains
some indecomposable non-projective modules that are not studied in
this paper (see for example their use in \cite{CGP2}).  Instead here,
we focus on semi-simple and projective modules that form together a
sub tensor category (see Proposition \ref{P:S@S}).

The category of $\Ubar$-modules has a grading in the abelian group
$\C/2\Z$ and its non semi-simple part is concentrated in degree
$\wb0,\wb1$ blocks. These two blocks form a category similar to the
category $\UqSm$-mod of representations of the standard small quantum
group $\UqSm$.  The category $\UqSm$-mod, equivalent to that of
modules over the triplet vertex operator algebra $\mathcal W(p)$ (see
\cite{K1991,NT}), has been intensively studied in logarithmic
conformal field theories (CFT) associated to the $(1,p)$ triplet
algebras (see \cite{FGST2006, FGST2006b, BFGT2009, BGT2011, CRW}).  In
particular, some results of Section \ref{S:projectives} are similar to
the analysis of projective modules in \cite{FGST2006}.

The category $\Ubar$-mod has additional modules which do not appear in
the representation theory of the small quantum group (in particular,
the one dimensional invertible objects mentioned above).  Moreover,
conjecturally $\Ubar$-mod is equivalent to the category of
representation of the vertex operator algebra called singlet vertex
algebra $\mathcal W(2, 2p-1)$ (see \cite{AM,tCM}).  Understanding a
deeper connection between the representation theory of this paper and
CFT deserves some attention.  For example, 
 it would be interesting to compare the CFT representations of $SL(2,\Z)$ (see
\cite{FGST2006b}) and more generally mapping class group representations (see
\cite{FSS}) with those obtained from $\Ubar$ in the TQFT of \cite{BCGP}.

 \subsection{Acknowledgements} 
 We would like to thank Simon Wood and Antun Milas for their useful
 comments on the relations with the theory of logarithmic CFTs and the
 organizers of the conference ``Modern trends in topological quantum
 field theory'' at the Erwin Schr\"odinger Institut (Vienna) for their
 kind invitation to the conference.

\section{A quantization of $\slt$ and its associated ribbon
  category}\label{S:QUantSL2H} 
In this section we recall the algebra $\Ubar$ and the category of
modules over this algebra.  Fix a positive integer $r$.  Let $r'=r$ if $r$ is odd and $r'=\frac{r}{2}$ else.  Let $\C$ be
the complex numbers and $\Cp=(\C\setminus \Z)\cup r\Z.$ Let
$q=e^\frac{\pi\sqrt{-1}}{r}$ be a $2r^{th}$-root of unity.  We use the
notation $q^x=e^{\frac{\pi\sqrt{-1} x}{r}}$.  For $n\in \N$, we also set 
 $$\qn{x}=q^x-q^{-x},\quad\qN{x}=\frac{\qn x}{\qn1},\quad\qn{n}!=\qn{n}\qn{n-1}\cdots\qn{1}\et\qN{n}!=\qN{n}\qN{n-1}\cdots\qN{1}$$

\subsection{The Drinfel'd-Jimbo quantum group} 
Let $\Uq$ be the $\C$-algebra given by generators $E, F, K, K^{-1}$ 
and relations:
\begin{align}\label{E:RelDCUqsl}
  KK^{-1}&=K^{-1}K=1, & KEK^{-1}&=q^2E, & KFK^{-1}&=q^{-2}F, &
  [E,F]&=\frac{K-K^{-1}}{q-q^{-1}}.
\end{align}
The algebra $\Uq$ is a Hopf algebra where the coproduct, counit and
antipode are defined by
\begin{align}\label{E:HopfAlgDCUqsl}
  \Delta(E)&= 1\otimes E + E\otimes K, 
  &\varepsilon(E)&= 0, 
  &S(E)&=-EK^{-1}, 
  \\
  \Delta(F)&=K^{-1} \otimes F + F\otimes 1,  
  &\varepsilon(F)&=0,& S(F)&=-KF,
    \\
  \Delta(K)&=K\otimes K
  &\varepsilon(K)&=1,
  & S(K)&=K^{-1}
.\label{E:HopfAlgDCUqsle}
\end{align}
Let $\UqMed$ be the algebra $\Uq$ modulo the relations
$E^\ro=F^\ro=0$.  Also, let $\UqSm$ be the algebra $\UqMed$ modulo the
relations $K^{2r}=1$.  These relations generate Hopf ideals so
$\UqMed$ and $\UqSm$ inherit a Hopf algebra structure.

As we will now explain, the categories of modules over $\Uq, \UqMed $
and $ \UqSm$ have very different properties.  Let $X$-mod be the
tensor category of finite dimensional $X$-modules for $X$ equal to
$\Uq, \UqMed$ or $\UqSm$.  The algebra $\UqSm$ is known as the small
quantum group and has been well studied, see \cite{CP} and the references within. 
The algebra $\Uq$ is known 
as the De Concini-Kac quantum group.  It and the category $\Uq$-mod have
rich structures and have been studied in \cite{DK, DKP, DKP2, DPRR}.
This category is not braided nor semi-simple and has an infinite
number of simple modules called cyclic modules which are not highest
weight modules.  Finally, the category $\UqMed$-mod is not semi-simple nor
braided and has an infinite number of non-isomorphic simple modules.
However, one can easily modify $\UqMed$ and obtain a braided category
of highest weight modules which has been used to construct invariants of 
links (\cite{GPT}), of 3-manifolds (\cite{CGP1}) and TQFTs (\cite{BCGP}).  The aim of this paper is to give an overview of the
algebraic results related to this modified quantization and prove a
few straightforward results.
  
 \subsection{A modified version of $\Uq$}  Let $\UsltH$ be the 
$\C$-algebra given by generators $E, F, K, K^{-1}, H$ and
relations in Equation \eqref{E:RelDCUqsl} plus the relations:
\begin{align*}
  HK&=KH, 
& [H,E]&=2E, & [H,F]&=-2F. 
\end{align*} 
The algebra $\UsltH$ is a Hopf algebra where the coproduct, counit and 
antipode are defined by Equations
\eqref{E:HopfAlgDCUqsl}--\eqref{E:HopfAlgDCUqsle} and by
\begin{align*}
  \Delta(H)&=H\otimes 1 + 1 \otimes H, 
  & \varepsilon(H)&=0, 
  &S(H)&=-H.
\end{align*}
Define $\Ubar$ to be the Hopf algebra $\UsltH$ modulo the relations
$E^\ro=F^\ro=0$.  

Let $V$ be a finite dimensional $\Ubar$-module.  An eigenvalue
$\lambda\in \C$ of the operator $H:V\to V$ is called a \emph{weight}
of $V$ and the associated eigenspace is called a \emph{weight space}.
A vector $v$ in the $\lambda$-eigenspace  
of $H$ is a \emph{weight vector} of \emph{weight} $\lambda$, i.e. $Hv=\lambda
v$.  We call $V$ a \emph{weight module} if $V$ splits as a direct sum
of weight spaces and $\qr^H=K$ as operators on $V$, i.e. $Kv=q^\lambda
v$ for any vector $v$ of weight $\lambda$.  Let $\cat$ be the category
of finite dimensional weight $\Ubar$-modules.

\begin{rem}  
The algebra $\Ubar$ does not have a requirement on $K^r$, allowing modules in $\cat$ to have non-integral weights.  The requirement $E^\ro=F^\ro=0$ forces modules to be highest weight modules.  As we will see the generator $H$ is used to define a braiding on $\cat$.  Here the main point is that one must know the action of $H$ and not just the action of $K$ which acts as a kind of exponential of $H$.  
\end{rem}

Since $\Ubar$ is a Hopf algebra then $\cat$ is tensor category where
the unit $\unit$ is the 1-dimensional trivial module $\C$.  Moreover,
$\cat$ is $\C$-linear: hom-sets are $\C$-modules, the composition and
tensor product of morphisms are $\C$-bilinear, and
$\End_\cat(\unit)=\C\Id_\unit$.  When it is clear we denote the unit
$\unit$ by $\C$.  We say a module $V$ is \emph{simple} if has no 
proper submodules.  If $V$ is simple then Schur's lemma implies that
$\End_\cat(V)=\C\Id_V$.  If $\End_\cat(V)=\C\Id_V$ then for
$f\in \End_\cat(V)$ we denote $\brk{f}$ as the scalar determined by
$f=\brk{f}\Id_V$.

We will now recall that the category $\cat$ is a ribbon category. 
  Let $V$ and $W$ be
objects of $\cat$.  Let $\{v_i\}$ be a basis of $V$ and $\{v_i^*\}$ be
a dual basis of $V^*=\Hom_\C(V,\C)$.  Then
\begin{align*}
  \coev_V :& \C \rightarrow V\otimes V^{*}, \text{ given by } 1 \mapsto \sum
  v_i\otimes v_i^*,  &
  \ev_V: & V^*\otimes V\rightarrow \C, \text{ given by }
  f\otimes w \mapsto f(w)
\end{align*}
are duality morphisms of $\cat$.  
In \cite{Oh} Ohtsuki truncates the usual formula of the $h$-adic
quantum $\slt$ $R$-matrix to define an operator on $V\otimes W$ by
\begin{equation}
  \label{eq:R}
  R=\qr^{H\otimes H/2} \sum_{n=0}^{\ro-1} \frac{\{1\}^{2n}}{\{n\}!}\qr^{n(n-1)/2}
  E^n\otimes F^n.
\end{equation}
where $q^{H\otimes H/2}$ is the operator given by  
$$q^{H\otimes H/2}(v\otimes v') =q^{\lambda \lambda'/2}v\otimes v'$$
for weight vectors $v$ and $v'$ of weights of $\lambda$ and
$\lambda'$. The $R$-matrix is not an element in $\Ubar\otimes \Ubar$,
however the action of $R$ on the tensor product of two objects of 
$\cat$ is a well defined linear map on such a tensor
product.  Moreover, $R$ gives rise to a braiding $c_{V,W}:V\otimes W
\rightarrow W \otimes V$ on $\cat$ defined by $v\otimes w \mapsto
\tau(R(v\otimes w))$ where $\tau$ is the permutation $x\otimes
y\mapsto y\otimes x$.
Also, let $\theta$ be the operator given by
\begin{equation}
\theta=K^{\ro-1}\sum_{n=0}^{\ro-1}
\frac{\{1\}^{2n}}{\{n\}!}\qr^{n(n-1)/2} S(F^n)\qr^{-H^2/2}E^n
\end{equation}
where $q^{-H/2}$ is an operator defined by on a weight vector $v_\lambda$ by
$q^{-H^2/2}.v_\lambda = q^{-\lambda^2/2}v_\lambda.$  
Ohtsuki shows that the family of maps $\theta_V:V\rightarrow V$ in
$\cat$ defined by $v\mapsto \theta^{-1}v$ is a twist (see
\cite{jM,Oh}).

 Now the ribbon structure on $\cat$ yields right duality morphisms 
\begin{equation}\label{E:d'b'}
  \tev_{V}=\ev_{V}c_{V,V^*}(\theta_V\otimes\Id_{V^*})\text{ and }\tcoev_V =(\Id_{V^*}\otimes\theta_V)c_{V,V^*}\coev_V
\end{equation}
which are compatible with the left duality morphisms $\{\coev_V\}_V$ and
$\{\ev_V\}_V$.  These duality morphisms are given by \begin{align*}
  \tcoev{V} :& \C \rightarrow V^*\otimes V, \text{ where } 1 \mapsto
  \sum K^{r-1}v_i \otimes v_i^*, \\ \tev_{V}: & V\otimes V^*\rightarrow
  \C, \text{ where } v\otimes f \mapsto f(K^{1-r}v).
\end{align*}
The \emph{quantum dimension} $\qdim(V)$ of an object $V$ in $\cat$ is the $\qdim(V)= \brk{\tev_V\circ \coev_V}=\sum  v_i^*(K^{1-r}v_i)$.  

 For $g\in\C/2\Z$, define
$\cat_{g}$ as the full sub-category of weight modules whose weights 
are all in the class  $g$ (mod $2\Z$).  
Then $\cat=\{\cat_g\}_{g\in \C/2\Z}$ is a $\C/2\Z$-grading (where
$\C/2\Z$ is an additive group): Let $V\in\cat_g$ and $V'\in\cat_{g'}$.
Then the weights of $V\otimes V'$ are congruent to $g+g' \mod 2\Z$,
and so the tensor product is in $\cat_{g+g'}$.  Also if $g\neq g'$ 
 then $\Hom_\cat(V, V')=0$ since morphisms in $\cat$  preserve weights.
Finally, for $f\in V^*=\Hom_\C(V,\C)$ then by definition the action of
$H$ on $f$ is given by $(Hf)(v)=f(S(H)v)=-f(Hv)$
and so
$V^{*}\in\cat_{-g}$.

\section{Modified traces on the projective modules.}
Let $\Proj$ be the full subcategory of $\cat$ consisting of projective
$\Ubar$-modules.  The subcategory $\Proj$ is an ideal (see also
\cite{GKP1}): it is closed under retracts (i.e.\ if $W \in \Proj$ and
$\alpha: X \to W$ and $\beta: W \to X$ satisfy $\beta \circ \alpha =
\Id_{X}$, then $X \in \Proj$) and if $X$ is in $\cat$ and $Y$ is in $\Proj$
then $X \otimes Y$ is in $\Proj$.

For any objects $V,W$ of $\cat$ and any endomorphism $f$ of $V\otimes
W$, set
\begin{equation}\label{E:trL}
\ptr_{L}(f)=(\ev_{V}\otimes \Id_{W})\circ(\Id_{V^{*}}\otimes
f)\circ(\tcoev_{V}\otimes \Id_{W}) \in \End_{\cat}(W),
\end{equation} and
\begin{equation}\label{E:trR}
\ptr_{R}(f)=(\Id_{V}\otimes \tev_{W}) \circ (f \otimes \Id_{W^{*}})
\circ(\Id_{V}\otimes \coev_{W}) \in \End_{\cat}(V).
\end{equation}

\begin{defi}\label{D:trace}  A \emph{trace on $\Proj$} is a family of linear functions
$$\{\mt_V:\End_\cat(V)\rightarrow K\}$$
where $V$ runs over all objects of $\Proj$ and such that the following two
conditions hold.
\begin{enumerate}
\item  If $U\in \Proj$ and $W\in \ob$ then for any $f\in \End_\cat(U\otimes W)$ we have
\begin{equation}\label{E:VW}
\mt_{U\otimes W}\left(f \right)=\mt_U \left( \ptr_R(f)\right).
\end{equation}
\item  If $U,V\in \Proj$ then for any morphisms $f:V\rightarrow U $ and $g:U\rightarrow V$  in $\cat$ we have 
\begin{equation}\label{E:fggf}
\mt_V(g\circ f)=\mt_U(f \circ g).
\end{equation} 
\end{enumerate}
\end{defi}

\section{The center of $\Ubar$}
The center of the small quantum group is known (see 
\cite{FGST2006b}) and its dimension is $3r-1$.  
The following proposition is a description of a subalgebra of the
center of $\Ubar$.
Let $C$ be \emph{quantum Casimir element} defined by
\begin{equation}\label{eq:casimir} 
  \Cas=FE+\dfrac{Kq+K^{-1}q^{-1}}{\qn{1}^2}=
  EF+\dfrac{Kq^{-1}+K^{-1}q}{\qn{1}^2}.
\end{equation} 
Also, let $\Tc_\ro$ be the $r$\textsuperscript{th} Chebyshev polynomial
determined by $\Tc_\ro(\frac{X+X^{-1}}2)= \frac{X^{\ro}+X^{-\ro}}2$.
\begin{prop}\label{P:PolMinC}
  The center of $\Ubar$ contains the $\C$-algebra generated by $C$ and $K^{\pm
    r}$ with the relation
  $\Tc_\ro\left(\frac{\qn1^{2}}2\Cas\right)=-\frac{K^r+K^{-r}}2$.
\end{prop}
\begin{proof}
  First, it is easy to see that the elements $C$ and $K^{\pm r}$ are
  central.  Next, we will show the relation stated in the proposition
  holds.  Using induction on $k\in\N$ one can show that
  \begin{equation}
    \label{eq:Pol-Om} 
   \prod_{i=0}^{k-1}\left(\Cas-   \dfrac{q^{-2i-1}K+q^{2i+1}K^{-1}}{\qn1^2}\right)=E^kF^k.
  \end{equation}
  
  On the other hand, we have
  $$2\bp{\Tc_\ro\bp{\frac{X+X^{-1}}2}-\Tc_\ro\bp{\frac{Y+Y^{-1}}2}}
  =\bp{X^r+X^{-r}}-\bp{Y^r+Y^{-r}}$$
$$=X^{-r}(X^r-Y^r)(X^r-Y^{-r})=\prod_{i=0}^{r-1}X^{-1}(X-q^{2i}Y)(X-q^{-2i}Y^{-1})$$
  $$
  =\displaystyle{\prod_{i=0}^{r-1}\bp{X+X^{-1}-Yq^{2i}-Y^{-1}q^{-2i}}}.$$
Combine the last expression with the fact that the product of Equation \eqref{eq:Pol-Om} vanishes for $k=r$ we obtain the following polynomial relation of degree $r$ for
  $\Cas$:
  $$2\Tc_\ro\left(\frac{\qn1^{2}}2\Cas\right)   -2\Tc_r\bp{\frac{qK+q^{-1}K^{-1}}{2}}=   \prod_{i=0}^{r-1}\left({\qn1^2}\Cas-  \bp{q^{2i+1}K+q^{-2i-1}K^{-1}}\right)  =0.$$
   Thus,  $\Tc_\ro\left(\frac{\qn1^{2}}2\Cas\right)=-\frac{K^r+K^{-r}}2$.
\end{proof}
It can be show that the center center of $\Ubar$ contains more complicated
elements involving the element $H$.  We don't need these elements in the
rest of the paper.
\section{Simple  $\Ubar$-modules}
For each $n \in \{0,\ldots,r-1\}$ let $S_n$ be the usual
$(n+1)$-dimensional simple highest weight $\Ubar$-module with
highest weight $n$. The module $S_n$ is a highest weight module with a highest weight vector $s_0$ such that $Es_0=0 $ and $Hs_0=ns_0$.  Then 
 $\{s_0, s_1,\ldots, s_n\}$  is a basis of $S_n$ where $Fs_i=s_{i+1}$, $H.s_i=(n-2i)s_i$, $E.s_0=0=F^{n+1}.s_0$
and $E.s_i=\frac{\qn i\qn{n+1-i}}{\qn1^2}s_{i-1}$.  The quantum
dimension $S_i$ is 
$\qdim(S_n)=(-1)^n\frac{\qn {n+1}}{\qn1}$.

In $\UqSm$-mod the modules $S_n$ are the only simple modules up to
isomorphism.  However in $\cat$, there is a $(n+1)$-dimensional simple
$\Ubar$-module with highest weight $n+r$ which not isomorphic to
$S_i$, as follows.  For $k\in \Z$, let $\C^H_{kr}$ be the one
dimensional modules where both $E$ and $F$ act by zero and $H$ acts by
$kr$.  The degree of $\C^H_{kr}$ is $kr\ {\rm mod} \ 2$.  Then 
$S_n\otimes \C^H_{kr}$ is the simple highest weight
module with highest weight $n+kr$.  As a $\UqSm$-module $\C^H_{kr}$ is 
isomorphic to the trivial module.  The modules $\C^H_{kr}$ are 
important tools in the work of \cite{CGP1, BCGP}.  
We also use another notation to distinguish among these modules,
those that are in the degree $\wb0$ part of $\cat$: we define for any
$k\in\Z$,
\begin{equation}
  \label{eq:sigma}
  \sigma^k=\C^H_{2kr'}\in\cat_{\wb 0}\quad\text{ where }\quad r'=r/2\text{ if }r\in2\Z
  \quad\text{ and }\quad r'=r \text{ else.}
\end{equation}

Next we consider a larger class of finite dimensional highest weight modules:
for each $\alpha\in \C$ we let $V_\alpha$ be the $r$-dimensional
highest weight $\Ubar$-module of highest weight $\alpha + r-1$.  The
modules $V_\alpha$ has a basis $\{v_0,\ldots,v_{r-1}\}$ whose action is
given by
\begin{equation}\label{E:BasisV}
H.v_i=(\alpha + r-1-2i) v_i,\quad E.v_i= \frac{\qn i\qn{i-\alpha}}{\qn1^2}
v_{i-1} ,\quad F.v_i=v_{i+1}.
\end{equation}
For all $\alpha\in \C$, the quantum dimension of $V_\alpha$ is zero: 
$$\qdim(V_\alpha)= \sum_{i=0}^{r-1} v_i^*(K^{1-r}v_i)=
 \sum_{i=0}^{r-1} q^{(r-1)(\alpha + r-1-2i)} =
 q^{(r-1)(\alpha + r-1)}\frac{1-q^{2r}}{1-q^{2}}=0.$$

 We say $V_\alpha$ is \emph{typical} if $\alpha\in  (\C\setminus \Z)\cup r\Z$, otherwise it is   \emph{atypical}.  If $V_\alpha$ is typical then  it is simple, since it is generated by any of the  basis vectors $v_i$ (see Equation \eqref{E:BasisV}).  

\begin{defi}
  The character of a weight module $V\in\cat$ is $\chi(V)=\sum_\alpha
  \dim(V(\alpha))X^\alpha\in\Z[\C]$ where $V(\alpha)$ is the
  $\alpha$-eigenspace of the action of $H$ on $V$ and $X^\alpha$ is a
  notation for the element $\alpha\in\C$ seen in the group ring
  $\Z[\C]$.
\end{defi}
Let $\qN k_{X}=X^{k-1}+X^{k-3}+\cdots +X^{-(k-1)}$.  Then for
$\alpha\in\Cp$ and $i\in\{0,\ldots,r-1\}$, one has
\begin{equation}
  \label{eq:caracSi}
  \chi(V_\alpha)=X^\alpha\qN r_X\et\chi(S_i)=\qN {i+1}_X.
\end{equation}

Let $V,W\in \cat$ and define
\begin{equation}
  \label{eq:Phi}
  \Phi_{V,W}=(\Id_{W}\otimes \tev_{V})\circ(c_{V,W}\otimes \Id_{V^*})\circ(c_{W,V}\otimes \Id_{V^*})\circ(\Id_W\otimes \coev_{V})
   \in \End(W).
\end{equation}

\begin{theo} \label{T:CgenSS} 
(1)  If $\wb\alpha\in \C/2\Z\setminus \Z/2\Z$ then $\cat_{\wb\alpha}$ is
  semi-simple.
  \\
(2) If $\alpha,\beta \in \Cp=(\C\setminus \Z)\cup r\Z $
  and $\alpha+\beta \notin \Z$ then $V_\alpha\otimes
  V_\beta\simeq\oplus_{k\in \Hr} V_{\alpha+\beta+k}$ where
  $\Hr=\{-(r-1),-(r-3),\ldots, r-1\}$.  \\
  (3) All the typical
  modules are projective.
 \end{theo}
\begin{proof}
   Let $\alpha\in\C\setminus \Z$ and define
  $c_\alpha=\frac{q^{\alpha+r}+q^{-\alpha-r}}{\qn1^2}$.  Proposition
  \ref{P:PolMinC} implies that $C$ satisfies  the
  relation $\prod_{i=0}^{r-1}\left(\Cas-c_{\alpha+2i}\right) =0$ on $\cat_{\wb\alpha}$.  
  Since
  $c_{\alpha+2i}-c_{\alpha+2j}=\frac{\qn{i-j}\qn{\alpha+r+i+j}}{\qn1^2}$,
  this polynomial has only simple roots.  Hence any
  $W\in\cat_{\wb\alpha}$ splits as the direct sum of the eigenspaces
  of $C$.  It is enough to show that $W$ is semi-simple when $C$
  acts by a scalar (say $c_\alpha$) on $W$.  Let $V$ be a maximal
  semi-simple submodule of $W$ and suppose $V\neq W$.  The weights of
  $W$ differ by elements of $2\Z$.  In particular, they are totally
  ordered and there is a weight vector $w$ of $W\setminus V$ of
  maximal weight $\lambda$.  Hence $E.w\in V$ and
  $FE.w=\bp{C-\dfrac{Kq+K^{-1}q^{-1}}{\qn{1}^2}}.w=0$ because it is
  proportional to $w$ and also in $V$.  It follows that
  $\lambda=\alpha+r-1$ modulo 2.  Then by Equation \eqref{eq:Pol-Om},
  $E^{r-1}F^{r-1}.w=\nu w$ where $\nu=\prod_{i=1}^{r-1}(c_{\alpha}-c_{\alpha-2i})\neq 0.$  Thus, $E.w= \frac1\nu
  E^{r}F^{r-1}.w=0$ and $w$ is an highest weight vector.  It follows that  $w$
  generates a module $V'$ isomorphic to the simple module
  $V_{\lambda-r+1}$ where $V\cap V'=\{0\}$.   This contradicts the
  maximality of $V$ and so $V=W$.

  The direct sum decomposition of $V_\alpha\otimes V_\beta$ follows
  from a straightforward calculation using the character formula for a
  typical module.  Finally, we prove the last statement of the theorem
  in two cases: 1) if $V_\alpha$ is a typical module with $\alpha\in
  \C\setminus \Z$ then the previous parts of the theorem imply that
  $V_\alpha$ is projective.  2) If $V_\alpha$ is a typical module with
  $\alpha =rn$ then it can be shown (see Lemma
  \ref{lem:generalhopflinks}) that the morphism
  $\Phi_{V_\beta,V_{rn}}$ defined in \eqref{eq:Phi} is non-zero for
  any $\beta \in \C\setminus \Z$.
  This morphism can be decomposed into the composition $g\circ f$
  where $f:V_{rn}\to V_\beta\otimes V_{rn} \otimes V_\beta^*$ and
  $g:V_\beta\otimes V_{rn} \otimes V_\beta^*\to V_{rn}$ are the
  obvious morphisms.  But $V_\beta\otimes V_{rn} \otimes V_\beta^*$ is
  projective because it is of the form $V_\beta\otimes W$ with
  $V_\beta$ projective.  Furthermore, $\frac1{\brk{g\circ f}} (g\circ
  f)=\Id_{V_{rn}}$.  Since the class of projective modules is closed
  under retracts then $V_{rn}$ is projective.
\end{proof}

\begin{lemma}
  Every simple module of $\cat$ is isomorphic to exactly one of the
  modules in the list:
\begin{itemize}
\item  $S_n\otimes  \C^H_{kr}$, for $n=0,\cdots, r-2$ and $k\in \Z$,  
\item  $V_\alpha$  for  $\alpha\in(\C\setminus \Z)\cup
  r\Z$. 
  \end{itemize}
\end{lemma}
\begin{proof}
  Let $W$ be a simple $\Ubar$-module in $\cat$.  Then $W$ is uniquely
  determined, up to isomorphism, by its highest weight $\lambda \in
  \C$.  The lemma follows from the fact that the highest weight of
  modules in the above list is in bijection with elements of $\C$.  
\end{proof}
Note in the above lemma  the modules $\C^H_{kr}$ and $S_{r-1}\otimes\C^H_{kr}$ are obtained by the isomorphisms $ \C^H_{kr}\cong S_0\otimes  \C^H_{kr}$ and $S_{r-1}\otimes\C^H_{kr}\cong V_{kr}$, respectively.

\begin{theo}\label{T:UniqueTrace}
  There exists a unique trace on $\Proj$ up to multiplication by an
  element of $\C$.  In particular, there is a unique trace $\mt
  =\left\{\mt_{V} \right\}_{V\in \Proj }$ on $\Proj$ such that for any
  $f\in\End_\cat(V_0)$ we have $\mt_{V_0}(f)=(-1)^{r-1}\brk{f}$.
\end{theo}
\begin{proof}
  The proof follows from results of \cite{GKP1,GPT}.  Here we explain
  this proof without recalling the definitions given in these papers:
  In \cite{GPT} we show that if $\alpha \in \C \setminus \frac12\Z$
  then $V_\alpha$ is an ambidextrous object in $\cat$.  In \cite{GKP1}
  we show that an ambidextrous object $J$ leads to  the existence of a unique (up to a constant) 
  trace $\mt$ on
  the ideal $\ideal_J$ generated by $J$.  When $J$ is simple then the
  trace is uniquely determined by the assignment $\mt_J(f)=c\brk{f}$,
  where $c$ is a constant.  Since $\Proj$ generated by any $V_\alpha$
  with $\alpha \in \C \setminus \frac12\Z$ then there exists a trace
  with the above property.  Finally, since $\Proj$ is generated by 
  $V_0$ the theorem follows. 
\end{proof}
We define the \emph{modified quantum dimension function} as  
$$ \md: \operatorname{Ob}(\Proj ) \to  K \;\; 
\text{ by } \md (V) = \mt_{V}\left(\Id_{V} \right).$$ 
We will prove in Lemma \ref{L:md} that the modified quantum dimension
function is given  by
\begin{equation}\label{eq:qd}
 \qd(V_\alpha)=(-1)^{r-1}\prod_{j=1}^{r-1}
\frac{\qn{j}}{\qn{\alpha+r-j}}=(-1)^{r-1}\frac{r\,\qn{\alpha}}{\qn{r\alpha}}={\color{green}
\frac{(-1)^{r-1}r}{q^{(1-r)\alpha} 
+\cdots+q^{(r-3)\alpha}+q^{(r-1)\alpha}}}
\end{equation}
for $\alpha\in\Cp$.

\section{Projective modules}\label{S:projectives}
Recall that an highest weight vector $v\in V$ is a weight vector such
that $Ev=0$.  We call a weight vector $v$ {\em dominant} if
$(FE)^2v=0$ (in particular, a highest weight vector is dominant).  
It is well known that a highest weight vector $v$ of a module $V$ generates a submodule with basis $\{F^iv\}$.  The following proposition describes the submodule generated by a dominant weight vector.  
\begin{prop}\label{P:DomV}
  Let $v\in V$ be a dominant vector of weight $i\in\{0,1,\ldots,r-2\}$
  and let $j=r-2-i$. Consider the following $2r$ vectors of $V$
  defined by 
  \begin{equation}\label{eq:Pi1}
   \vech_i=v, \qquad     \vecr_{r-j}=E\vech_i,\qquad \vecs_{i}=F\vecr_{r-j},
    \qquad \vecl_{j-r}=F^{i+1}\vech_i,
  \end{equation}
  \vspace*{-3.5ex}
  \begin{align}
    \vech_{i-2k}&=F^k\vech_i&&\et& \vecs_{i-2k}&=F^k\vecs_i&\text{ for }k=\{0\cdots i\},\label{eq:Pi2}\\
    \vecr_{r-j+2k}&=E^k\vecr_{r-j}&&\et& \vecl_{j-2k-r}&=F^k\vecl_{j-r}
    &\text{ for }k=\{0\cdots j\}.   \label{eq:Pi3}
  \end{align} 
  Then the vector space they generate is a submodule of $V$ and the
  following relations holds in $V$ (whenever the involved vectors are
  defined):
  \begin{align}
    H\vec^X_k&=k\vec^X_k,&K\vec^X_k&=q^k \vec^X_k\text{ for }X\in\{L,R,H,S\},\label{eq:Pi4}
    \\
    E\vecr_k&=\vecr_{k+2},&F\vec^X_k&=\vec^X_{k-2}\text{ for }X\in \{H,S,L\}, 
    \label{eq:Pi5}
    \\
    F\vech_{-i}&=\vecl_{j-r},& E\vecl_{j-r}&=\vecs_{-i}, \qquad
    E\vecr_{j+r}=E\vecs_i=F\vecs_{-i}=F\vecl_{-j-r}=0
  \end{align}
  \vspace*{-3.5ex}
  \begin{align}
  E\vech_{i-2k}&=\gamma_{i,k}\vech_{i-2k+2}+\vecs_{i-2k+2},&&&
  E\vecs_{i-2k}&=\gamma_{i,k}\vecs_{i-2k+2}
   \label{eq:Pi7}
  \\\label{eq:Pi8}
  F\vecr_{r-j+2k}&=-\gamma_{j,k}\vecr_{r-j+2k-2}  
  &&\text{and}& E\vecl_{j-2k-r}&=-\gamma_{j,k}\vecl_{j-2k-r+2}
  \end{align}
  where $\gamma_{n,k}=\qN{k}\qN{n-k+1}=\gamma_{n,n-k+1}$.
\end{prop}
\begin{figure}[H]
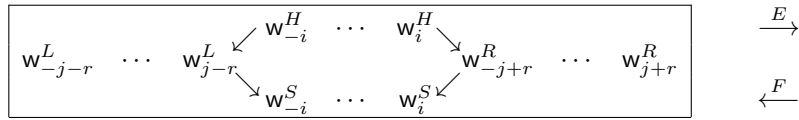

  \centering
 \[
\begin{array}{|ccccccccc|cc}
\cline{1-9}
  &&& \vech_{-i}& \cdots &\vech_i&&&&&\stackrel E{\longrightarrow}\\
\vecl_{-j-r}& \cdots & 
\vecl_{j-r}
&&&& \vecr_{-j+r}& \cdots &\vecr_{j+r}&\ \ \ &\\
  &&& \vecs_{-i}& \cdots &\vecs_i&&&&&\stackrel F{\longleftarrow}\\
\cline{1-9}
\end{array}
\put(-220,8){$\swarrow$}\put(-219,-9){$\searrow$}
\put(-143,8){$\searrow$}\put(-143,-9){$\swarrow$}
\]
 
  \caption{The weight spaces structure of the module $P_i$ (here $j=r-2-i$).}
  \label{fig:P_i}
\end{figure}
\begin{proof}
First, we show that the Relations \eqref{eq:Pi4}--\eqref{eq:Pi8} hold.   
The actions of $H$ and $K$ are easily deduced from their commutation
  relations with $E$ and $F$.  
  The formulas in \eqref{eq:Pi5} are restatements of 
  \eqref{eq:Pi2}--\eqref{eq:Pi3}.  
 
  Let
  $\cas_i=\frac{q^{i+1}+q^{-i-1}}{\qn1^2}=-\frac{q^{j+1}+q^{-j-1}}{\qn1^2}$.    
We have $\Cas\vech_i=\cas_i\vech_i+\vecs_i$.  Since $FE\vecs_i=0$ then
  $\Cas$ acts by the scalar $\cas_i$ on $\vecs_{i}$.  As $\Cas$ is central, we
  have
  $$E\vecl_{j-r}=EF\vech_{-i}=\bp{\Cas-\frac{Kq^{-1}+K^{-1}q}{\qn{1}^2}}F^i\vech_i=
  F^i\bp{\Cas-\cas_i}\vech_i=F^i\vecs_i=\vecs_{-i}.$$ 
  Next, $F\vecl_{-j-r}=F(F^{j}F^{i+1}\vech_i)=F^r\vech_i=0$ so $\Cas$ acts by $c_i$ on
  $\vecl_{-j-r}$. Then by induction on $k=0\cdots j-1$, 
  \begin{align*}
  E\vecl_{2k-j-r}&=EF\vecl_{2k-j-r+2}=
  \bp{\Cas-\frac{Kq^{-1}+K^{-1}q}{\qn{1}^2}}\vecl_{2k-j-r+2}\\
  &=\bp{\cas_i-\frac{q^{2k-j-r+1}+q^{-2k+j+r-1}}{\qn{1}^2}}\vecl_{2k-j-r+2}
  =-\qN{j-k}\qN{k+1}\vecl_{2k-j-r+2}\\
  &=-\gamma_{j,j-k}\vecl_{2k-j-r+2}.
  \end{align*}
  By writing $2k-j$ as $j-2(j-k)$ we obtain the second formula in \eqref{eq:Pi8}.
Moreover, $\Cas$ acts by $c_i$ on 
    $\vecl_{2k-j-r+2}=-\gamma_{j,j-k}^{-1}E\vecl_{2k-j-r}$.
  Now $\Cas$ also acts by $\cas_i$ on $\vecs_{i-2k}=F^k\vecs_i$.  This implies
  that for $k=1\cdots i$,
  \begin{align*}
    E\vecs_{i-2k}&=EF\vecs_{i-2k+2}
    =\bp{\Cas-\frac{Kq^{-1}+K^{-1}q}{\qn{1}^2}}\vecs_{i-2k+2}\\
    &=\bp{\cas_i-\frac{q^{i-2k+1}+q^{-i+2k-1}}{\qn{1}^2}}\vecs_{i-2k+2}
    =\qN{k}\qN{i-k+1}\vecs_{i-2k+2}.
  \end{align*}
Since $E^r=0$ we have $E\vecs_i=E\frac{(-1)^j}{\qN i!^2\qN j!^2}E^{r-1}\vecl_{-r-j}=0$.  
   This implies that $EF\vecr_{r-j}=E\vecs_i=0$, so $\Cas$ acts by the scalar $\cas_i$ on $\vecr_{r-j}$ and on $\vecr_{r-j+2k}=E^k\vecr_{r-j}$.     Using
  this, for $k=1\cdots j$ we have 
  \begin{align*}
    F\vecr_{r-j+2k}&=FE\vecr_{r-j+2k-2}=
    \bp{\Cas-\frac{Kq+K^{-1}q^{-1}}{\qn{1}^2}}\vecr_{r-j+2k-2}\\
    &=\bp{\cas_i-\frac{q^{r-j+2k-1}+q^{-r+j-2k+1}}{\qn{1}^2}}\vecr_{r-j+2k-2}=
    -\qN{k}\qN{j-k+1}\vecr_{r-j+2k-2}.
  \end{align*}
Since   $F^r=0$ we have $F\vecs_{-i}=F\frac{(-1)^j}{(\qN j!)^2}F^{r-1}\vecr_{r+j}=0$. 
 Then $\Cas-\cas_i$ sends $\vech_i\mapsto\vecs_i$ so it sends
  $\vech_{i-2k}=F^k\vech_i\mapsto\vecs_{i-2k}=F^k\vecs_i$.  This implies that for
  $k=1\cdots i$,
     \begin{align*}
    E\vech_{i-2k}&=EF\vech_{i-2k+2}=
    \bp{\Cas-\frac{Kq^{-1}+K^{-1}q}{\qn{1}^2}}\vech_{i-2k+2}\\
    &=\bp{\cas_i-\frac{q^{i-2k+1}+q^{-i+2k-1}}{\qn{1}^2}}\vech_{i-2k+2}+\vecs_{i-2k+2}
    =\qN{k}\qN{i-k+1}\vech_{i-2k+2}+\vecs_{i-2k+2}.
  \end{align*} 
  Now $F\vech_{-i}=FF^i\vech_{i}=\vecl_{j-r}$.  
  
  The final relation to check is $E\vecr_{j+r}=0$.  Relation \eqref{eq:Pi7} implies 
  $$E^i\vech_{-i}=c \vech_{i} +c'\vecs_{i}$$
where $c$ and $c'$ are   
constants and $c\neq0$.  
We have
$$E^r\vech_{-i}=E^{r-i}(c \vech_{i} +c'\vecs_{i})=cE^{r-i-1}\vecr_{r-j}=cE\vecr_{r+j}.$$
Thus, we have shown the Relations \eqref{eq:Pi1}--\eqref{eq:Pi8} hold.

Finally the vector space generated by the vectors of this proposition
is clearly stable by the generators of $\Ubar$ so it is a submodule of
$V$. 
   \end{proof}

Projective indecomposable weight modules in $\cat_{\wb0}\cup\cat_{\wb1}$ have a highest weight vector.  The following proposition classifies the isomorphism classes of these modules.  
\begin{prop}\label{P:proj-mod}
  Let $i\in\{0,1,\ldots,r-2\}$ and let $j=r-2-i$.  Denote the
  vectors of the canonical basis of $\C^{2r}$ by
  $$(\vech_i,\vech_{i-2},\ldots,\vech_{-i},\,\vecr_{r+j},\vecr_{r+j-2},\ldots,\vecr_{r-j},\,
  \vecl_{j-r},\vecl_{j-2-r},\ldots,\vecl_{-j-r},\,\vecs_i,\vecs_{i-2},\ldots,\vecs_{-i}).$$
  Then Formulas \eqref{eq:Pi1}-\eqref{eq:Pi8} define a structure of
  weight module on $\C^{2r}$  
  which we denote by $P_i$.  Here $P_{r-1}=S_{r-1}=V_0$. 
   The module $P_i$ is projective and indecomposable.   Any projective indecomposable weight module
  $P\in\cat_{\wb0}\cup\cat_{\wb1}$  
 with highest weight $(k+1)r-i-2$ is isomorphic to $P_i\otimes \C^H_{kr}$.
 \end{prop}
\begin{proof}
   A direct computation shows that the commutation relation
  $EF-FE=\frac{K-K^{-1}}{\qn1}$ is satisfied on $P_i$, the other
  relations are consequences of the fact that $E$ and $F$ translate
  the weight spaces (see Figure \ref{fig:P_i}).  Hence Formulas
  \eqref{eq:Pi1}-\eqref{eq:Pi8} define a structure of weight module on
  $P_i$.

Proposition \ref{P:DomV} implies $P_i$ is a module which is generated by its dominant vector $\vech_i$.  Furthermore, 
 if $V$ is any weight module then
  \begin{equation}\label{eq:dominant}
    \Hom(P_i,V)\simeq\{v\in V:v\text{ is dominant of weight }i\}.
  \end{equation}
  In particular, $\vech_i$ and $\vecs_i$ are both dominant vector of
  $P_i$ of weight $i$, thus $\End(P_i)$ is a two dimensional vector
  space generated by $\Id:\vech_i\mapsto\vech_i$ and the nilpotent map
  $x_i:\vech_i\mapsto\vecs_i$ given by the action of $\Cas-\cas_i$.
  This implies that $\End(P_i)$ is a local algebra and that $P_i$ is
  indecomposable (see \cite[section 5.2]{Pi}).
  
   Let now $\phi:V\to P_i$ be a surjective map in $\cat$. 
   As $\Cas$ is
  central, $V$ splits as the direct sum of the characteristic spaces
  of $\Cas$ and only the summand $V_i$ of $V$ associated to the
  eigenvalue $c_i$ is not included in $\ker\phi$.  We claim that
  $V_i\subset\ker(\Cas-\cas_i)^2_{|V}$.  
  Indeed Proposition \ref{P:PolMinC} implies that
  on a module of $\cat_{\wb0}$, 
    $\Tc_\ro\left(\frac{\qn1^{2}}2\Cas\right)=-1=\Tc_r(\frac{q+q^{-1}}2)$ so
    $\prod_{i=0}^{r-1}\bp{\Cas-\cas_{2i}}=0$ 
    and on a module of $\cat_{\wb1}$, we have
    $\Tc_\ro\left(\frac{\qn1^{2}}2\Cas\right)=1=\Tc_r(\frac{1+1}2)$ so
    $\prod_{i=0}^{r-1}\bp{\Cas-\cas_{2i-1}}=0.$ 
    In both cases, all roots of the minimal polynomial of $\Cas$ have
    multiplicity at most $2$ and this proves that
    $V_i\subset\ker(\Cas-\cas_i)^2_{|V}$.  
Any vector $v\in V_i$ of weight $i$ satisfy
    $(\Cas-\cas_i)^2.v=(FE)^2.v=0$ thus is dominant. 
      Let
    $v\in\phi^{-1}(\{\vech_i\})$.  Then Proposition \ref{P:DomV} implies that
    there is an unique $\psi\in\Hom(P_i,V_i)$ sending $\vech_i\mapsto v$.
    Furthermore, $\phi\circ\psi(\vech_i)=\vech_i$ so $\phi\circ\psi=\Id_{P_i}$.
    Thus $\phi$ has a section and $P_i$ is projective.
    
    For all $0\leq i\leq r-1$, $\Hom(P_i,S_i)\simeq\{v\in S_i:v\text{ is dominant of
      weight }i\}$ is one dimensional generated by 
    the surjective morphism $\pi_i$
    sending $\vech_i$ to a highest weight vector $v_i$ of $S_i$.  Let $P'$ be an indecomposable projective module and $\phi':P'\to
    S$ a surjective map to a simple module (obtained by taking the quotient of 
    $P'$ by a maximal submodule).  Then for some $k\in\Z$ and some $0\leq i\leq r-1$, there exist an isomorphism $S\otimes\C^H_{rk}\simeq S_i$.  If
    $i=r-1$ then $S$ is simple and projective. In this case $\phi'$ has a section and is
    an isomorphism since $P'$ is indecomposable.    Assume now $i<r-1$
    and let $P=P'\otimes \C^H_{rk}$ which is also projective and
    indecomposable.  Let $\phi=\phi'\otimes\Id_{\C^H_{rk}}:P\to S_i$.  
       Since $P$ 
    is projective, there exists $\psi:P\to P_i$ such that
    $\phi=\pi_i\circ \psi$.  Let $v\in P$ such that $\phi(v)=v_i\in
    S_i$.  Then $\psi(v)\in\pi_i^{-1}(v_i)\cap
    P_i(i)=\vech_i+\C\vecs_i$ (here $P_i(i)$ is the weight space of
    weight $i$ of $P_i$). 
    Hence $\psi(v)$ generates $P_i$ and $\psi$
    is surjective.  Finally, as $P_i$ is projective, $\psi$ has a
    section and is an isomorphism since $P$ is indecomposable.  Thus
    we have $P'=P\otimes\C^H_{-rk}\simeq P_i\otimes\C^H_{-rk}$.
\end{proof}

The module $P_i$ is also an injective module in $\cat$ which is self
dual.  Let $i\in\{0,\dots, r-2\}$. After identifying both $S_i$ and $P_i$
with their duals, the map $\pi_i^*:S_i\to P_i$ is an injective morphism
with image $\Span_\C(\vecs_i,\ldots,\vecs_{-i})$.  The quotient
$(\ker\pi_i)/\pi_i^*(S_i)$ is isomorphic to
$(\C^H_r\oplus\C^H_{-r})\otimes S_j$ where $j=r-i-2$.  As in the proof of Proposition \ref{P:proj-mod}, we let $x_i$ be the nilpotent endomorphism of $P_i$,  that sends
$\vech_i\mapsto\vecs_i$.  We have
$$\End(P_i)=\C \Id\oplus \C x_i=\C[x_i]/(x_i^2).$$  
Finally, the character of $P_i$ is given by
$$\chi(P_i)=2[i+1]_X+(X^r+X^{-r})[r-i-1]_X=[r]_X(X^{r-i-1}+X^{-r+i+1}).$$

\begin{cor} For all $\alpha\in\Cp$, $V_\alpha\otimes
  V_{-\alpha}$ is isomorphic to $ V_0\otimes V_0$.
\end{cor}
\begin{proof}
  It follows from the previous proposition that the projective modules of
  $\cat$ are determined up to isomorphism by their characters, so $V_\alpha\otimes
  V_{-\alpha}$ and $ V_0\otimes V_0$ are isomorphic because they have the same
  character.
\end{proof}
Recall the definitions of $\sigma^k$ and $r'$ in Equation \eqref{eq:sigma}.
\begin{cor} \label{C:factor V0^2}
  Let $\alpha\in\C\setminus\Z$.  Let $P\in\cat_{\wb 0}$ be a
  projective module, then there exist maps $f_i:P\to
  V_0\otimes\sigma^{n_i}\otimes V_0^*$,
  $g_i:V_0\otimes\sigma^{n_i}\otimes V_0^*\to P$, $f'_j:P\to
  V_{\alpha+2k_j}\otimes V_{-\alpha}$ and $g'_j:V_{\alpha+2k_j}\otimes
  V_{-\alpha}\to P$ such that
  $$\Id_P=\sum_ig_if_i+\sum_jg'_jf'_j$$
  where $n_i\in\Z$ and $k_j\in\Z\setminus r'\Z$. 
\end{cor}
\begin{proof}
  Consider the epimorphism $f=\Id_P\otimes\ev_{V_{-\alpha}}:P\otimes
  V_{-\alpha}^*\otimes V_{-\alpha}\to P$.  Since $P$ is projective
  this morphism has a left inverse $g$, i.e. $f\circ g=\Id_P$.  Now
  $P\otimes V_{-\alpha}^*\in\cat_{\wb{\alpha+r-1}}$ splits as a direct
  sum of modules isomorphic to $V_{\alpha+2k}$ $(k\in\Z)$.  This
  produces a factorization of $\Id_P$ through the modules
  $V_{\alpha+2k}\otimes V_{-\alpha}$.  Finally, if $k=r'n\in r'\Z$,
  then $V_{\alpha+2k}\otimes V_{-\alpha}\simeq \sigma^n\otimes
  V_{\alpha}\otimes V_{-\alpha}\simeq \sigma^n\otimes V_{0}\otimes
  V_{0}\simeq V_{0}\otimes \sigma^n\otimes V_{0}^*$.
\end{proof}

The modified trace is non degenerate:
\begin{prop}\label{P:qtnondegen}
  Let $P\in\cat$ be a projective module and $V\in\cat$.  Then
  the pairing  
  $$
  \begin{array}{ccl}
    \Hom_\cat(V,P)\times\Hom_\cat(P,V)&\to&\C\\
    (h_1,h_2)&\mapsto&\qt_P(h_1h_2)
  \end{array}
  $$
  is non degenerate.
\end{prop}
\begin{proof}
  Let $\alpha\in\C\setminus\Z$.  Let $f:P\to V_\alpha\otimes W$ and
  $g:V_\alpha\otimes W\to P$ be morphisms of $\cat$ such that
  $gf=\Id_P$.  We show that for any non zero $h:P\to V$, there exists
  $h':V\to P$ such that $\qt_P(h'h)\neq0$.  Indeed, we have
  $h=hgf\neq0$ thus we have a non trivial morphism
  $(hg\otimes\Id_{W^*})\circ(\Id_{V_\alpha}\otimes\coev_W):V_\alpha\to
  V\otimes W^*$.  But $V_\alpha$ is in the category $\cat_{\wb \alpha}$ which is
  semi-simple so the previous map has a left inverse $k: V\otimes
  W^*\to V_\alpha$.  Then we have that
  
  $$\qt_{V_\alpha}(k\circ(hg\otimes\Id_{W^*})\circ(\Id_{V_\alpha}\otimes\coev_W))
  =\qd(V_\alpha)\neq0.$$
  Let $k'=(k\otimes\Id_W)\circ(\Id_V\otimes\tcoev_W):V\to
  V_\alpha\otimes W$ and $h'=gk'$.  Finally, 
 $$\qt_P(h'h)=\qt_P(g(k'h))=\qt_{V_\alpha\otimes   W}((k'h)g)=\qt_{V_\alpha}(k\circ(hg\otimes\Id_{W^*})\circ(\Id_{V_\alpha}\otimes\coev_W))  =\qd(V_\alpha)\neq0.$$
\end{proof}

\newcommand{\qx}[2]{\frac{\qn{#1}}{\qn{#2}}}
\begin{lemma}[General Hopf links]\label{lem:generalhopflinks}
Recall the map $\Phi$ given in Equation \eqref{eq:Phi}.  For all $ i,j\in \{0,1,\ldots , r-2\}$ and $
  \alpha,\beta\in \Cp=(\C\setminus \Z)\cup r\Z$, one has 
  \begin{align}\label{eq:HopfSimple1}
    \Phi_{V_\beta,V_\alpha}&=\frac{(-1)^{r-1}r}{\qd(V_\alpha)}q^{\alpha\beta}\Id_{V_\alpha}&
    \Phi_{S_i,S_j}&=(-1)^i\qx{(i+1)(j+1)}{j+1}\Id_{V_j}
    \\\label{eq:HopfSimple2}
    \Phi_{S_i,V_\alpha}&=\qx{(i+1)\alpha}{\alpha}\Id_{V_\alpha}&
    \Phi_{P_i,V_\alpha}
    &=(-1)^{r-1}r\frac{q^{(r-1-i)\alpha}+q^{-(r-1-i)\alpha}}{\qd(V_\alpha)}\Id_{V_\alpha}
  \end{align}
      Moreover, recall the nilpotent $x_j\in\End(P_j)$ given by the action of
    $\Cas-\cas_j$, then 
{\color{black} \begin{align}
    \Phi_{S_i,P_j}&=(-1)^i\qx{(i+1)(j+1)}{j+1}\Id_{P_j}+
    (-1)^i\qn{1}^2\frac{i\qn{(i+2)(j+1)}-(i+2)\qn{i(j+1)}}{\qn{j+1}^3}\,x_j,
  \end{align}}
{\color{black}   \begin{align}
   \Phi_{V_0,P_j}&=(-1)^{r+j}\frac{2r\qn1^2}{\qn{j+1}^2}x_j,&
\Phi_{P_i,P_j}&=\frac{(-1)^{i}2r\qn1^2}{\qn{j+1}^2}\bp{q^{(i+1)(j+1)}+q^{-(i+1)(j+1)}}x_j.
  \end{align}} 
\end{lemma}
\begin{proof}
  For $\alpha\in\C$, let $\Psi_\alpha:\Z[\C]\to \C$ be the map sending
  $X^z\mapsto q^{\alpha z}$.  We start by observing the following
   fact: Let $w$ be a highest weight vector of $W$ of weight
  $\alpha$, then
  \begin{equation}\label{eq:hopflinkshw}
    \Phi_{V,W}(w)=\Psi_{\alpha+1-r}(\chi(V))w.
  \end{equation}
    Indeed, the map
  $\Phi_{V,W}$ is given by the partial quantum trace of $c_{V,W}\circ
  c_{W,V}$.  A standard argument shows that on a highest weight
  vector, this partial trace only depends of the Cartan part
  $q^{H\otimes H/2}$ of the $R$-matrix.  The identity then follows from a
  direct computation. A detailed presentation of an analogous
  computation is given in \cite[Proposition 2.2]{GP08}.  Equation \eqref{eq:hopflinkshw} implies that  if $W$ is simple then 
  $\Phi_{V,W}=\Psi_{\alpha+1-r}(\chi(V))\Id_W$.

  Equations \eqref{eq:HopfSimple1} and \eqref{eq:HopfSimple2} follow
  from Equation \eqref{eq:hopflinkshw}.  For example,
  $\Phi_{P_i,V_\alpha}=\lambda\Id_{V_\alpha}$ where
  $$\lambda=\Psi_{\alpha}(\chi(P_i))=\Psi_{\alpha}\bp{[r]_X(X^{r-i-1}+X^{-r+i+1})}
  =\frac{(-1)^{r-1}r}{\qd(V_\alpha)}(q^{(r-1-i)\alpha}+q^{-(r-1-i)\alpha}).$$
 
 Similarly, to compute $\Phi_{S_1,P_j}$ observe that any endomorphism of $P_j$
  is of the form $a\Id+b x_j\in \End(P_j)=\C[x_j]/(x_j^2)$.  Computing as
  above,  
  $$\Phi_{S_1,P_j}(\vecs_j)= a\vecs_j $$
  where $\vecs_j$ is the highest weight vector of $P_j$ and $a=\Psi_{j+1-r}(X+X^{-1})=-(q^{j+1}+q^{-j-1})$.  
   We now compute $b$.
  \newcommand{\vs}{{\mathsf s}}
  Recall that $S_1$ is generated by two weight vectors $\vs_0,\vs_1$ and
  $E.\vs_1=\vs_0$, $F.\vs_0=\vs_1$, $H.\vs_i=(-1)^i \vs_i$.  In general, 
  $$c_{W,V}=\tau\circ R =  \tau\circ q^{H\otimes H/2}(\Id \otimes \Id + 
  (q-q^{-1})E\otimes F + \cdots ).$$ 
  So
 $$c_{S_1,P_j}\circ c_{P_j,S_1}(\vech_{j}\otimes \vs_0)
 =c_{S_1,P_j}(q^{\frac{j}{2}}\vs_0\otimes \vech_{j}
 +q^{-\frac{j+2}{2}}(q-q^{-1})\vs_1\otimes \vecr_{j+2})$$
 $$=\bp{q^{j}\vech_{j}+(q-q^{-1})^2q^{-1} \vecs_j}\otimes \vs_0+\cdots\otimes \vs_1$$
 $$c_{S_1,P_j}\circ c_{P_j,S_1}(\vech_{j}\otimes \vs_1)
 =c_{S_1,P_j}(q^{-\frac{j}{2}}\vs_1\otimes
 \vech_{j})=q^{-j}\vech_{j}\otimes \vs_1+\cdots\otimes \vs_0.$$ When
 taking the quantum trace with respect to $S_1$ (i.e. the trace on
 $S_1$ of the endomorphism composed with $\Id\otimes K^{1-r}$, we then
 get that 
 $$\Phi_{S_1,P_j}(\vech_{j})=-q\bp{q^{j}\vech_{j}+(q-q^{-1})^2q^{-1} \vecs_j}
 -q^{-1}.q^{-j}\vech_{j}=-(q^{j+1}+q^{-j-1})\vech_{j}-(q-q^{-1})^2 \vecs_j$$
 and we get $a=-(q^{j+1}+q^{-j-1})$, $b=-(q-q^{-1})^2$.  We have 
 $$\Phi_{S_1,P_j}\circ\Phi_{S_i,P_j}=\Phi_{S_1\otimes
   S_i,P_j}=\Phi_{S_{i+1},P_j}+\Phi_{S_{i-1},P_j}$$
so  $\Phi_{S_i,P_j}$ is determined by the recurrence relations 
 $$\Phi_{S_{i+1},P_j}=(a+bx_j)\Phi_{S_i,P_j}-\Phi_{S_{i-1},P_j}, \,\quad\Phi_{S_1,P_j}
 =a+bx_j\quad \text{and} \quad\Phi_{S_0,P_j}=1.$$
Solving for $\Phi_{S_i,P_j}$ we have the unique solution
{\color{black}$$\Phi_{S_i,P_j}=\frac{(-1)^i}{\qn{j+1}}\bp{\qn{(i+1)(j+1)}\Id_{P_j}
  +\frac{\qn1^2x_j}{\qn{j+1}^2}\,\big({i\qn{(i+2)(j+1)}-(i+2)\qn{i(j+1)}}\big)}.$$ }
In particular, for $i=r-1$, $S_i=V_0$ and we get 
{\color{black}$$\Phi_{V_0,P_j}=(-1)^{r+j}\frac{2r\qn1^2}{\qn{j+1}^2}x_j.$$}

Finally the character formulas give the isomorphism of projective modules: 
$$V_0\otimes S_{r-i-1}=V_0\otimes S_{r-i-3}\oplus P_i.$$
Thus, 
{\color{black}
\begin{align*}
  \Phi_{P_i,P_j}&=\frac{(-1)^{i}2r\qn1^2}{\qn{j+1}^2}\bp{q^{(i+1)(j+1)}+q^{-(i+1)(j+1)}}x_j.
\end{align*}}
\end{proof}

\begin{lemma}\label{L:Phicom} 
  If $P$ is a projective module then
  $\qt_P(\Phi_{V_0,P})=\qt_{V_0}(\Phi_{P,V_0})=(-1)^{r-1}
  \brk{\Phi_{P,V_0}}$.
\end{lemma}
\begin{proof} 
  From the
  properties of a trace in Definition \ref{D:trace} we have
  $$\mt_P(\Phi_{V_0,P})=\mt_P(\ptr_R(c_{P,V_0}c_{V_0,P}))
  =\mt_{P\otimes V_0}(c_{V_0,P}c_{P,V_0})=\mt_{V_0\otimes P}(c_{P,V_0}c_{V_0,P})$$
  $$=\mt_{V_0}(\ptr_R(c_{P,V_0}c_{V_0,P}))=\mt_{V_0}(\Phi_{P,V_0})
  =(-1)^{r-1}
  \brk{\Phi_{P,V_0}}$$
  where the last equality follows from Theorem \ref{T:UniqueTrace}.
\end{proof}

\begin{lemma}[The modified trace on typical modules]\label{L:md}
  Let $V_\alpha$ be a typical module.  Then 
  for any $f\in\End_\cat(V_\alpha)$,
  $\tr_{V_\alpha}(f)=\md(V_\alpha)\brk{f}$ where $\md({V_\alpha})$ is
  given in Equation \eqref{eq:qd}.
\end{lemma}
\begin{proof}
  First, since $V_\alpha$ is simple we have
  $\mt_{V_\alpha}(f)=\brk{f}\mt_{V_\alpha}(\Id_{V_\alpha})=\qd(V_\alpha)\brk{f}$
  where $\qd(V_\alpha)=\mt_{V_\alpha}(\Id_{V_\alpha})$.  From Lemma
  \ref{L:Phicom} we have 
  $\qt_{V_\alpha}(\Phi_{V_0,V_\alpha})=\qt_{V_0}(\Phi_{V_\alpha,V_0})$.  
  \\
  So 
  $\md(V_\alpha)\brk{\Phi_{V_0,V_\alpha}}=
  \md(V_0)\brk{\Phi_{V_\alpha,V_0}}$ where 
{\color{black}$\md(V_0)=(-1)^{r-1}$}, and 
  $$ \md(V_\alpha)=
  \md(V_0)\frac{\brk{\Phi_{V_\alpha,V_0}}}{\brk{\Phi_{V_0,V_\alpha}}}.$$
  Finally, the formula for $\md(V_\alpha)$ follows from Lemma
  \ref{lem:generalhopflinks}. 
\end{proof}

\begin{lemma}[The modified trace on $P_j$]\label{lem:modifiedtraces}
We have 
$$\qd(P_j)=\qt_{P_j}(\Id_{P_j})=(-1)^{j+1}(q^{j+1}+q^{-j-1})\et
{\color{black}\qt_{P_j}(x_j)=(-1)^{j+1}\frac{\qn{j+1}^2}{\qn1^2}}.$$
\end{lemma}
\begin{proof}  
  As $V_0$ is projective, so are $V_0\otimes S_{r-j-1}$ and
  $V_0\otimes S_{r-j-3}$.  Now by Proposition \ref{P:proj-mod},
  indecomposable projective modules are determined by their highest
  weight so the isomorphism class of a projective module is determined
  by its character.  Hence the character formulas imply that there
  exists an isomorphism of projective modules $V_0\otimes
  S_{r-j-1}\simeq V_0\otimes S_{r-j-3}\oplus P_j$.  Taking the
  modified traces of the identities of these modules gives
  $\qd(V_0)\qdim(S_{r-j-1})=\qd(V_0)\qdim(S_{r-j-3})+\qd(P_j)$.  Since
  $\qdim(S_i)=(-1)^{i}\frac{\{i+1\}}{\{1\}}$ we have
  \begin{align*}
    \qd(P_j)&=\qd(V_0)\bp{\qdim(S_{r-j-1})-\qdim(S_{r-j-3})}
    =(-1)^{j}\bp{\qN{r-j}-\qN{r-j-2}}\\&=(-1)^{j+1}\bp{q^{j+1}+q^{-j-1}}.
  \end{align*}
  
 Lemma \ref{L:Phicom} implies $\qt_{P_j}(\Phi_{V_0,P_j})=\qt_{V_0}(\Phi_{P_j,V_0})
.$
  Then Lemma \ref{lem:generalhopflinks}, implies
   $$(-1)^{r+j}2r\qt_{P_j}(x_j){\color{black}\frac{\qn1^2}{\qn{j+1}^2}}=(-1)^{r-1}r\frac{2}{\qd(V_0)}\qt_{V_0}(\Id_{V_0})$$
 which implies the second relation of the lemma. 
\end{proof}

\begin{lemma}[Twist on $P_j$]\label{lem:twistPj}
  The action of the twist on $P_j$ is given by 
  {\color{black}  $$\theta_{P_j}=(-1)^jq^{\frac{j^2+2j}2}(1-(r-j-1)\frac{\qn1^2}{\qn{j+1}} x_j).$$}
  In particular, $\theta_{P_j}$ has infinite order.
\end{lemma}
\begin{proof}
  The twist commutes with the map $\pi_j:P_j\to S_j$ and $\pi_jx_j=0$.  Thus
  the twist on $P_j$ is given by $\theta_{P_j}=\theta_{S_j}(1+\lambda x_j)$ where
  $\theta_{S_j}=(-1)^jq^{\frac{j^2+2j}2}$ is the scalar action of the twist on $S_j$.
  Hence 
  {\color{black}  $$\qt(\theta_{P_j})=-q^{\frac{j^2+2j}2}(q^{j+1}+q^{-j-1}+\lambda\frac{\qn{j+1}^2}{\qn1^2}).$$}
  Finally we use again the module $V_0\otimes S_{r-j-1}\simeq
  V_0\otimes S_{r-j-3}\oplus P_j$ to color the unknot {\color{black} with} framing $+1$.
  Its double is the Hopf link with both component{\color{black}s} having framing $+1$
  and this gives
  $$\theta_{V_0}\theta_{S_{r-j-1}}.\qt(\Phi(S_{r-j-1},V_0))
  =\theta_{V_0}\theta_{S_{r-j-3}}.\qt(\Phi(S_{r-j-3},V_0))+\qt(\theta_{P_j})$$
  Hence 
  \begin{align*}
    \qt(\theta_{P_j})&=\theta_{V_0}\qd(V_0)
    \bp{(r-j)\theta_{S_{r-j-1}}-(r-j-2)\theta_{S_{r-j-3}}}\\
    &=-q^{j^2/2-1}((r-j)-(r-j-2)q^{-2j-2})\\
    &=-q^{\frac{j^2+2j}2}(-(r-j-2)q^{j+1}+(r-j)q^{-j-1})
  \end{align*}
  and this gives the announced formula for $\lambda$.
\end{proof}

\section{The algebra of projective modules}\label{S:AlgProj}
In this section, we define and study two algebras encoding the maps
between projective modules of $\cat_{\wb 0}$ and $\cat_{\wb 1}$
respectively.   
These are the algebras one would associate to curves in a $1+1+1$-TQFT which would be an extension of the $2+1$ TQFT given in  \cite{BCGP}. 
\subsection{Maps between indecomposable projective modules}
We first describe the maps between indecomposable projective modules in the categories $\cat_{\wb0}$ and $\cat_{\wb1}$:
\begin{prop}\label{P:indecMaps}
  Let $i,\ell\in\{0,\dots, r-2\}$ and $k\in\Z$.  Let
  $P=\C^H_{kr}\otimes P_\ell$ be an indecomposable module, then any
  non zero map $P_i\to P$ is equal to $\lambda I_i+\mu x_i$, $\lambda
  \alpha_i^+$ or $\lambda \alpha_i^-$ where $\lambda,\mu\in\C$ and the
  maps $I_i$, $x_i$, $\alpha_i^+$ and $\alpha_i^-$ are uniquely
  determined by
$$
\begin{array}[t]{rcl}
  I_{i}:P_i&\to& P_i\\
  \vech_i&\mapsto&\vech_i
\end{array}
\quad
\begin{array}[t]{rcl}
  x_{i}:P_i&\to& P_i\\
  \vech_i&\mapsto&\vecs_i
\end{array}
$$$$
\begin{array}[t]{rcl}
  \alpha^+_{i}:P_i&\to& \C^H_{r}\otimes P_{{r-2-i}}\\
  \vech_i&\mapsto&1\otimes\vecl_{i-r}
\end{array}
\quad
\begin{array}[t]{rcl}
  \alpha^-_{i}:P_i&\to& \C^H_{-r}\otimes P_{{r-2-i}}.\\
  \vech_i&\mapsto&[i]!^{-2}1\otimes\vecr_{i+r}
\end{array}
$$
\end{prop}
\begin{proof}
  By Equality \eqref{eq:dominant}, the space $\Hom_\cat(P_i,P)$ is
  isomorphic to the space of dominant weight vectors of weight $i$ of
  $P$.  Now the space of dominant vectors of $\C^H_{kr}\otimes P_\ell$
  has dimension $4$ and is generated by $1\otimes\vech_\ell$ and
  $1\otimes\vecs_\ell$ of weight $\ell+kr$, $1\otimes\vecl_{-\ell-2}$
  of weight $-\ell-2+kr$ and $1\otimes\vecr_{2r-\ell-2}$ of weight
  ${(k+2)r-\ell-2}$.  The result then follows by analyzing for which
  $k,\ell$ the module $\C^H_{kr}\otimes P_\ell$ has dominant weight
  vectors of weight $i$.
\end{proof}
Tensoring by $\C^H_{nr}$ gives canonical isomorphisms
$\Hom_\cat(P_i,\C^H_{kr}\otimes P_\ell)\cong
\Hom_\cat(\C^H_{nr}\otimes P_i,\C^H_{(n+k)r}\otimes P_\ell)$.  Then
for $i\in\{0,\dots, r-2\}$ 
and $j=r-2-i$, maps between indecomposable projective modules
$P_i^{k}=\C^H_{kr}\otimes P_i$, $P_j^{k}=\C^H_{kr}\otimes P_j$ can be
represented by the following periodic quiver: 
$$\cdots\,\, \epsh{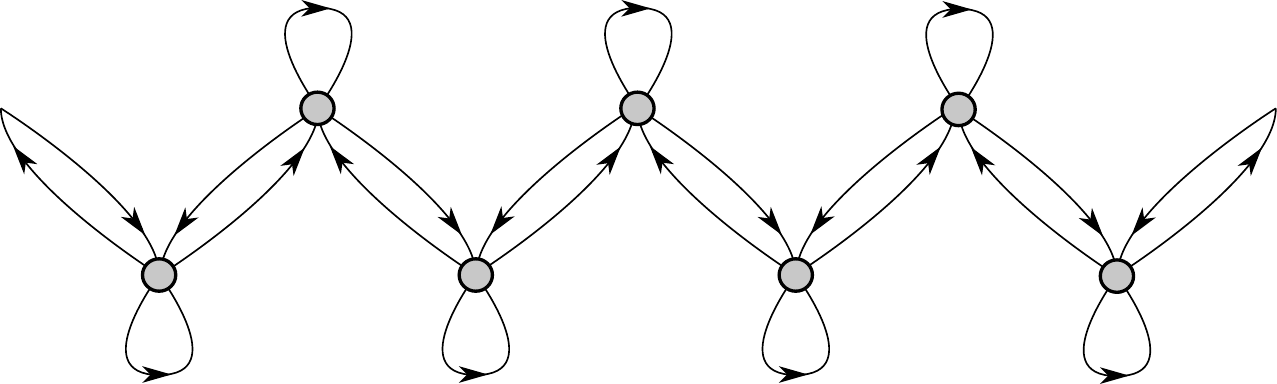}{15ex}
\put(-66,17){\ms{P_i^2}}\put(-120,17){\ms{P_i}}\put(-181,17){\ms{P_i^{-2}}}
\put(-95,-15){\ms{P_j^{1}}}\put(-40,-15){\ms{P_j^{3}}}
\put(-152,-15){\ms{P_j^{-1}}}\put(-206,-15){\ms{P_j^{-3}}}
\,\,\cdots$$

\subsection{The algebras of curves}
\renewcommand{\PP}{{\mathbb P}}
\renewcommand{\AA}{{\mathbb A}}
As above, let $r'=r$ if $r$ is odd and $r'=\frac{r}{2}$ else.  Let
$\sigma=\C^H_{2r'}$ be the one dimensional module where $E$ and $F$
act as $0$ and $H$ acts as $2r'$.  
The object $\sigma\in\cat_{\wb 0}$ generates the group of invertible
objects of $\cat_{\wb 0}$ which is isomorphic to $\Z$.  For $k\in\Z$,
we just denote by $\sigma^k$ the module $\C^H_{2kr'}$ so that
$\sigma^k\otimes\sigma^\ell=\sigma^{k+\ell}$ and $\sigma^0=\unit$.

A $\sigma$-invariant module is an infinite dimensional weight
module $V$ with finite dimensional weight spaces and with the property that
$\sigma\otimes V=V$.  Then tensor product by $1\in\sigma$ gives an
action of $\Z\simeq\{\sigma^k:k\in\Z\}$ on $V$ denoted by
$v\mapsto\sigma v$. Remark that since $V$ is infinite dimensional it is not an object of $\cat$. 

Let $\cat^\sigma$ be the category whose objects are $\sigma$-invariant
modules and maps from $V$ to $W$ are given by the set
$\Hom_\sigma(V,W)$ of morphism of $\UsltH$-modules that commute with
the action of $\sigma$.

We study endomorphisms of the $\sigma$-invariant module
$$\PP=\bigoplus_{k\in\Z}\bigoplus_{i=0}^{r-1}\C^H_{kr}\otimes P_i.$$  
According to the parity of weights, this module splits into two $\sigma$-invariant modules $\PP=\PP_{\wb 0}\oplus\PP_{\wb 1}$.  Also for
$\nu\in\{\wb0,\wb1\}$, $\PP_\nu$ has a $\Z$-grading for which
$\C^H_{kr}\otimes P_i$ is of degree $k$.  In particular, if $r$ is odd, the
action of $\sigma=\C^H_{2r}$ shifts the degree by $2$ whereas for $r$
even,
$\sigma=\C^H_r$ shifts the degree by $1$.  In the standard way, the $\Z$-grading of
$\PP_\nu$ turns
$\End_\sigma(\PP_\nu)$ into a $\Z$-graded algebra.  We call 
$$\AA_{\wb0}=\End_\sigma(\PP_{\wb0})\et \AA_{\wb1}=\End_\sigma(\PP_{\wb1})$$
these $\Z$-graded algebras.
\\

We now introduce two algebras $A$ and $B$ used to describe
$ \AA_{\nu}=\End_\sigma(\PP_\nu)$.  Let $A$ be the algebra of graded dimension
$2s^{-1}+4+2s$ which is the quotient of the $\C$-path algebra associated to
the quiver
$$\Gamma= \quad \epsh{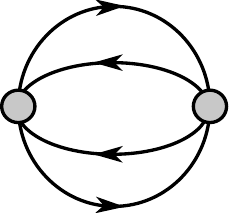}{13ex}\put(-58,2){\ms{p}}\put(-7,2){\ms{q}}
\put(-33,33){\ms{a_+}}\put(-33,-30){\ms{a_-}}
\put(-33,17){\ms{b_+}}\put(-33,-16){\ms{b_-}}\quad\text{ by the
  relations}\left\{
  \begin{array}{l}
    a_+b_+=b_+a_+=a_-b_-=b_-a_-=0\,,\\
    a_+b_-+a_-b_+=0\,,\\
    b_+a_-+b_-a_+=0\,.
  \end{array}
\right.$$ As a $\Z$-graded $\C$-vector space, $A$ is spanned in degree
$1$ by $\{a_+,b_+\}$, in degree $-1$ by $\{a_-,b_-\}$ and in degree $0$
by $\{p,q=1-p,x=b_+a_-,y=a_+b_-\}$.

The algebra $B$ is the quotient of $A$ obtained by identifying $p=q$,
$a_+=b_+$ and $a_-= b_-$ ($B$ is the exterior algebra of $\C^2$).  It
is also the quotient of the $\C$-path algebra associated to the quiver
$$\Gamma'= \quad \epsh{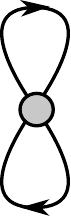}{13ex}\put(-10,1){\ms{p}}
\put(-13,22){\ms{a_+}}\put(-13,-19){\ms{a_-}}
\quad\text{ by the
  relations}\left\{
  \begin{array}{l}
    a_+a_+=a_-a_-=0\,,\\
    a_+a_-+a_-a_+=0\,.
  \end{array}
\right.$$
The basis of $B$ is given by $\{a_-,p=1,x=a_+a_-,a_+\}$.

\begin{theo}\label{T:EndP}
  There exist isomorphisms of algebras:
  \begin{align}
  \text{If $r\in 3+2\N$, }&\AA_{\wb0}=\End_\sigma(\PP_{\wb0})\simeq
    A^{\frac{r-1}2}\times\C\simeq \End_\sigma(\PP_{\wb1})=\AA_{\wb1}.
    \label{T:EndP1}\\
  \text{If $r\in 2+4\N$, }&\AA_{\wb0}=\End_\sigma(\PP_{\wb0})\simeq
    A^{\frac{r-2}4}\times B\et\AA_{\wb1}=\End_\sigma(\PP_{\wb1})\simeq
    A^{\frac{r-2}4}\times \C.
    \label{T:EndP2}\\
  \text{If $r\in 4+4\N$, }&\AA_{\wb0}=\End_\sigma(\PP_{\wb0})\simeq
    A^{\frac{r}4}\et\AA_{\wb1}=\End_\sigma(\PP_{\wb1})\simeq
    A^{\frac{r-4}4}\times B\times \C.
    \label{T:EndP3}
  \end{align}
\end{theo}
\begin{proof}
  To prove this theorem we build the explicit isomorphisms.  If a
  $\sigma$-invariant module $W$ splits as
  $W=\bigoplus_{k\in\Z}\sigma^k\otimes V$ for some finite dimensional
  weight module $V$ then the action of $\sigma$ on $W$ is free and the
  restriction map $\Hom_\sigma(W,W')\to\Hom_{\UsltH}(V,W')$ 
  is easily seen to be an isomorphism (here $\Hom_{\UsltH}$ denotes
  morphisms of ${\UsltH}$-modules).  Using this fact, we restrict our
  study to the maps from $P_i$ to $\PP$.  For $i\le r-2$, let
  $j=r-2-i$.  By Proposition \ref{P:indecMaps}, the space
  $\Hom_{\UsltH}(P_i,\PP)$ is of dimension 4 generated by the
  morphisms determined uniquely by
  $$
  \begin{array}[t]{rcl}
    I_{i}:P_i&\to& P_i\subset\PP 
  \end{array}
  \quad
  \begin{array}[t]{rcl}
    x_{i}:P_i&\to& P_i\subset\PP 
  \end{array}
  $$$$
  \begin{array}[t]{rcl}
    \alpha^+_{i}:P_i&\to& \C^H_{r}\otimes P_{j}\subset\PP
      \end{array}
  \quad
  \begin{array}[t]{rcl}
    \alpha^-_{i}:P_i&\to& \C^H_{-r}\otimes P_{j}\subset\PP
      \end{array}
  $$
  These maps extend to maps of $\End_\sigma(\PP)$ on factors
  $\C^H_{kr}\otimes P_i$ by tensoring them by the identity of
  $\C^H_{kr}$ and we extend them by $0$ on the other factors.  We use
  the same name for these extended maps of $\End_\sigma(\PP)$.  The
  composition of these maps is computed by looking at the image of the
  dominant vector $\vech_i\in P_i$.  One easily gets
  $$\alpha^+_{j}\circ\alpha^+_{i}=0=\alpha^-_{j}\circ\alpha^-_{i}.$$
  Now we use that if $v\in V$ is a weight vector 
  then in $\C^H_{\pm r}\otimes V$ one has $E.(1\otimes v)=1\otimes (E.v)$
  and $F.(1\otimes v)=1\otimes (-F.v)$ to compute:
  $$
  \begin{array}{rcccl}
    P_i&\stackrel{\alpha_i^+}{\longrightarrow}&\C^H_r\otimes P_j&\stackrel{\alpha_j^-}{\longrightarrow}&P_i\\
    \vech_i&\longmapsto&1\otimes\vecl_{i-r} 
    =(-F)^{j+1}(1\otimes \vech_j)&\longmapsto& [j]!^{-2}(-F)^{j+1}\vecr_{j+r}
  \end{array}
  $$
  As
  $F^{j+1}\vecr_{j+r}=\prod_{k=1}^j(-\gamma_{j,j-k+1})\vecs_i=(-1)^j[j]!^2\vecs_i$,  we
  get $$\alpha_j^-\circ\alpha_i^+=-x_{i}.$$ Similarly,
  $$
  \begin{array}{rcccl}
    P_i&\stackrel{\alpha_i^-}{\longrightarrow}&\C^H_{-r}\otimes P_j&\stackrel{\alpha_j^+}{\longrightarrow}&P_i\\
    \vech_i&\longmapsto&[i]!^{-2}1\otimes\vecr_{i+r}=[i]!^{-2}E^{i+1}(1\otimes \vech_j)&\longmapsto& [i]!^{-2}E^{i+1}\vecl_{j-r}
  \end{array}
  $$
  And as
  $E^{i+1}\vecl_{j-r}=\prod_{k=1}^i(\gamma_{i,k})\vecs_i=[i]!^2\vecs_i$,
  we get $$\alpha_j^+\circ\alpha_i^-=x_{i}.$$

  We now explicit the isomorphism of Theorem \ref{T:EndP}.  First
  remark that the maps of $\End_\sigma(\PP)$ commute with $K^r$ thus
  they restrict to maps of $\End_\sigma(\PP_\nu)$ for
  $\nu\in\{\wb0,\wb1\}$.  Next the decomposition of endomorphism
  algebras in Theorem \ref{T:EndP} follows from the fact that these
  endomorphisms respect the characteristic spaces of the Casimir
  element $\Cas$ whose minimal polynomial is given in Proposition
  \ref{P:PolMinC}.  For $i\in\{0\cdots r-1\}$, let
  $\cas_i=\frac{q^{i+1}+q^{-i-1}}{\qn1^2}=-\frac{q^{j+1}+q^{-j-1}}{\qn1^2}$
  be the scalar by which $\Cas$ acts on the simple module $S_i$.  The
  action of $\Cas$ on $\sigma\otimes S_i$ and on $S_i$ are the same if
  $r$ is odd, but they are opposite if $r$ is even.

  Let $\nu\in\{0,1\}$.  For $i\in2\N+\nu$, $i\le r'-2$, the kernel of
  $(\Cas^2-\cas_i^2)^2$ on $\PP_\nu$ is
  $V=\bigoplus_{k\in\Z}\sigma^k\otimes (P_i\oplus Q_j)$ where $Q_j$ is
  $P_j$ if $r$ even and $Q_j=\C^H_r \otimes P_j$ if $r$ is odd.  Then
  an isomorphism $A\stackrel\sim\to\End_\sigma(V)$ is given by
  \begin{align*}
    p&\mapsto I_{i} & x&\mapsto x_{i} & a^+&\mapsto \alpha^+_{i}
    & a^-&\mapsto \alpha^-_{i}\\
    q&\mapsto I_{j} & y&\mapsto x_{j} & b^+&\mapsto \alpha^+_{j} &
    b^-&\mapsto \alpha^-_{j}
  \end{align*}
  Now if $r$ is even, let $i=\frac{r-2}2=r-2-i$ and $\nu=i $ mod
  $2$. Then $\cas_i=0$ and the kernel of $\Cas^2$ on $\PP_\nu$ is then
  $V=\bigoplus_{k\in\Z}\sigma^k\otimes P_i$.  Then an isomorphism
  $B\stackrel\sim\to\End_\sigma(V)$ is given by
  \begin{align*}
    p&\mapsto I_{i} & x&\mapsto x_{i} & a^+&\mapsto \alpha^+_{i} &
    a^-&\mapsto \alpha^-_{i}
  \end{align*}
  Finally the remaining $\C$ factors in Theorem \ref{T:EndP}
  correspond to the eigenspace of $\Cas$ associated to the simple
  eigenvalue $\cas_{r-1}$.
\end{proof}

\newcommand{\Tr}{\operatorname{Tr}}
\newcommand{\STr}{\operatorname{STr}} In the paper \cite{BHLZ}, the
concepts of Coend, trace and the Hochschild-Mitchell homology in a
linear category are related.  In \cite{BCGP}, a graded TQFT is defined
for manifolds equipped with a 1-cohomology class with value in
$\C/2\Z$.  The algebras $\AA_\nu$ would naturally be associated to a
curve $\gamma$ with cohomology class $\coh$ such that
$\coh([\gamma])=\nu$.  Then the graded vector space
$\Tr(\AA_\nu)=(\AA_\nu)_{/fg=gf}$ maps surjectively onto the TQFT
space of the torus $\gamma\times S^1$ with cohomology class $\coh$
such that $\coh([\gamma\times *])=\nu$ and $\coh([*\times S^1])=0$.
Here we define a graded version of the trace of $\AA_\nu$ that
surjects on the TQFT space of the torus $\gamma\times S^1$ with
cohomology class $\coh$ such that $\coh([\gamma\times *])=\nu$ and
$\coh([*\times S^1])=\beta$ for any $\beta\in\C/2\Z$ (instead of
$\beta$, we use $z=q^{2r'\beta}$).

\newcommand{\cro}[1]{\left[#1\right]}
\newcommand{\Al}{\mathcal A}
Let $z\in\C^*$, $\Al$ be a $\Z$-graded $\C$-algebra.  If $f,g$ are
homogenous elements of degree $|f|,|g|\in\Z$, let
$\cro{f,g}_z=fg-z^{|f|}gf$.  Define the $\Z$-graded 
module
$$\Tr^z(\Al)=\Al_{/[\Al,\Al]_z}.$$
Similarly, if $\Al$ is considered as a super algebra, the bracket is
replaced by $\cro{f,g}_z^-=fg-(-1)^{|f||g|}z^{|f|}gf$ and we define the
$\Z$-graded super module
$$\STr^z(\Al)=\Al_{/[\Al,\Al]_z^-}.$$
\begin{prop} Recall the algebras $A$ and $B$ above. Then
  \begin{enumerate}
  \item If $z\in\C^*\setminus\{\pm 1\}$, $\Tr^z(A)\simeq\C^2\simeq\STr^z(A)$.
  \item $\Tr^{\pm1}(A)\simeq\C^3\simeq\STr^{\pm1}(A)$.
  \item  If $z\neq 1$ then  $\STr^z(B)\simeq \C$ and  $\STr^1(B)\simeq B$.
  \end{enumerate}
  Here the spaces $\C^2$ and $\C^3$ are concentrated in degree $0$.  As
  a consequence, we have the following graded dimensions: 
  \begin{enumerate}
  \item If $r\in2\Z+1$ and $z\neq\pm1$, then 
    $$\dim_s(\Tr^z(\AA_{\wb0}))=\dim_s(\Tr^z(\AA_{\wb 1}))=r\et
    \dim_s(\Tr^{\pm1}(\AA_{\wb 0}))=\dim_s(\Tr^{\pm1}(\AA_{\wb 1}))=\frac{3r-1}2.$$
  \item If $r\in4\Z+2$ and $z\neq\pm1$, then 
    $$\dim_s(\STr^z(\AA_{\wb 0}))=\dim_s(\STr^z(\AA_{\wb 1}))=\frac r2,$$
    $$\dim_s(\STr^{\pm1}(\AA_{\wb 1}))=\dim_s(\STr^{-1}(\AA_{\wb 0}))
    =\frac{3r-2}4\et$$
    $$\dim_s(\STr^{1}(\AA_{\wb 0}))=s^{-1}+\frac{3r+2}4+s$$
  \end{enumerate}
  where $\dim_s$ is the sum for $k\in\Z$ of $s^k$ times the dimension
  of the degree $k$ subspace.
\end{prop}
\begin{proof}
  Let $\ve=\pm1$. First remark that for any elements $f,g$ of the
  algebra,
  $\cro{f,g}_z^\ve+\ve^{|f|.|g|}z^{|f|}\cro{g,f}_z^\ve=(1-z^{|f|+|g|})fg$.
  Hence if $ z^{|f|+|g|}\neq1$ then $fg=0$ in the quotient, else
  $\cro{g,f}_z^\ve$ and $\cro{f,g}_z^\ve$ are proportional.  Finally
  for $g=1$, one gets that a map $f$ vanishes in the quotient unless
  $z^{|f|}=1$. Then the relations in $A$ implies that
  $\cro{A,A}^\ve_z$ is generated by the following elements
  \begin{itemize}
  \item $\cro{a_\pm,p}_z^\ve=a_\pm$,
  \item $\cro{b_\pm,q}_z^\ve=b_\pm$,
  \item $\cro{b_+,a_-}_z^\ve=b_+a_--\ve z a_-b_+=x+\ve z y$
  \item $\cro{a_+,b_-}_z^\ve=a_+b_--\ve z b_-a_+=y+\ve z x$
  \end{itemize}
  If $z^2\neq1$ then $x=y=0$ in the quotient, and if $z=\pm1$,
  then $\cro{A,A}^\ve_z$ is generated in degree $0$ by $x+\ve z y$.

  Similarly for $z\neq1$, $\cro{B,B}_z^-$ is generated by $a_+$, $a_-$
  and the element 
  $$\cro{a_+,a_-}_z^-=a_+a_-+ z a_-a_+=(1-z)x.$$ 
  On the other hand, for
  $z=1$ we have $\cro{B,B}^-_1=0$.

  For the last statements, we use
  $\Tr^z(\Al\times\Al')=\Tr^z(\Al)\oplus\Tr^z(\Al')$ and
  $\STr^z(\Al\times\Al')=\STr^z(\Al)\oplus\STr^z(\Al')$.
\end{proof}

\section{Decomposition of tensor products}\label{S:DecOfTensorPro} 
We recall the different notations for the simple self-dual projective module:
$$P_{r-1}=V_0=S_{r-1}.$$
From Proposition \ref{P:proj-mod} any projective indecomposable module
of $\cat_{\wb0}\cup\cat_{\wb1}$ is an element of the set
$\{P_i\otimes\C^H_{kr}, i\in \{0,1,\ldots r-1\}, k\in \Z\}$.   
Let us recall their characters 
$$\chi(P_i\otimes \C^H_{kr})=X^{kr}[r]_X(X^{r-i-1}+X^{-r+i+1})\et 
\chi(V_0\otimes \C^H_{kr})=X^{rk}[r]_X$$
where  $i\in \{0,1,\ldots r-2\}$.
Observe now that these characters are linearly independent in
$\Z[X^{\pm 1}]$ and form a basis of an ideal of polynomials which are
divisible by $[r]_X$ (but not of the whole ideal generated by
$[r]_X$).

As a consequence to decompose a projective module $P$ in direct sum of
projective indecomposable ones, it is sufficient to decompose
$\chi(P)$ as $$\chi(P)=\sum_{i=0}^{r-1}\sum_{k_i\in\Z} n_{i,k_i} 
\chi\bp{\C^H_{k_ir}\otimes P_{i}}.
$$
\newcommand{\Sumto}[2]{{\displaystyle{\sum_{
        \tiny{\begin{array}{c}
          #1\\\text{by }2
        \end{array}}}^{#2}\!\!\!}}}
\newcommand{\Sumtom}[2]{{\displaystyle{\sum_{
        \tiny{\begin{array}{c}
          #1\\\text{by }-2
        \end{array}}}^{#2}\!\!\!}}}
\newcommand{\Opto}[2]{{\displaystyle{\bigoplus_{
        \tiny{\begin{array}{c}
          #1\\\text{by }2
        \end{array}}}^{#2}\!\!\!}}}
\newcommand{\Optom}[2]{{\displaystyle{\bigoplus_{
        \tiny{\begin{array}{c}
          #1\\\text{by }-2
        \end{array}}}^{#2}\!\!\!}}}
In the following, we write 
$$\Sumto{k=m}n \et \Opto{k=m}n$$ 
for the sums where $k$ is $k\le n$ and varies in the set  $ m+2\N$.  Similarly, we write 
 $$\Sumtom{k=n}m \et \Optom{k=n}m$$ 
for the sums where $k\ge m$ and varies in the set $ n-2\N$.
\begin{lemma}[Decomposition of tensor products $V_0\otimes S_i$] 
  Let $0\leq i\le r-1$.   
 Then $$V_0\otimes S_i=\Opto{k=r-1-i}{r-1} P_{k}.$$ 
\end{lemma}
\begin{proof}
If $i$ is odd it holds :
$$\chi(V_0\otimes S_i)=[r]_X[i+1]_X=[r]_X\Sumtom{j=i}{1} (X^{j}+X^{-j})
=\Sumtom{j=i}{0} \chi(P_{r-1-j}).$$
If $i$ is even it holds :
$$\chi(V_0\otimes S_i)=[r]_X[i+1]_X=[r]_X\bp{1+\Sumtom{j=i}{2} (X^{j}+X^{-j})}
=\Sumtom{j=i}{0} \chi(P_{r-1-j}).$$
\end{proof}

\begin{prop}[The decomposition of the tensor products $P_i\otimes S_j$]\label{prop:decoPiSj}
  Let  $0\leq i\le r-2$ and $0\le j\le r-1$. It holds: 
  $$P_i\otimes S_j=\left(\Opto{k=|i-j|}{\min(i+j,r-1)} P_{k}\right)\oplus 
  \left(\Opto{k=2r-2-i-j}{r-1} P_{k}\right)\oplus
  \left(\Opto{k=r+i-j}{r-1} P_{k}\otimes (\C^{H}_{r}\oplus
    \C^{H}_{-r})\right)$$
  where   the sums are meant to be empty if the lower bound is bigger than the
  upper bound. 
\end{prop}
\begin{proof}
It holds : 
$$\chi(P_i\otimes S_j)=[r]_X(X^{r-i-1}+X^{-r+i+1})[j+1]_X=[r]_X([r-i+j]_X-[r-i-j-2]_X).$$ 
Recall that $[-n]=-[n]$.  We denote the parity of $r-1-i+j$ and $r-3-i-j$  by $p\in \{0,1\}$ (note they coincide).
If $i>j$ and $i+j\leq r-2$ we have :
$$\chi(P_i\otimes S_j)=[r]_X\left(\Sumto{l=r-i-j-1}{r-i+j-1}(X^l+X^{-l})\right)=\Sumto{l=r-i-j-1}{r-i+j-1}\chi(P_{r-1-l})=\Sumto{k=i-j}{i+j} \chi(P_k).$$
If $i\geq j$ and $i+j\geq r-1$ and $p=1$  
we have :
\begin{align*}
\chi(P_i\otimes S_j)=[r]_X\left(\Sumto{l=p}{r-i+j-1}(X^l+X^{-l})+\Sumto{l=p}{i+j-r+1}(X^l+X^{-l})\right)=\\
=\Sumto{l=p}{r-i+j-1}\chi(P_{r-1-l})+\Sumto{l=p}{i+j-r+1}\chi(P_{r-1-l})=\Sumto{k=i-j}{r-1-p}\chi(P_k)+\Sumto{k=2r-2-i-j}{r-1-p}\chi(P_k).
\end{align*}
Note a similar calculation gives the result above in the case $p=0$ (just pay attention to the fact that if $p=0$ the terms $X^0+X^{-0}$ should be replaced by $X^0$).

Let us now suppose that $j >i$ and $i+j\leq r-2$ and let $q\in \{0,1\}$ be the parity of $j-i-1$. Then, if $q=1$ (as above, a similar calculation proves the same final formula if $q=0$) it holds : 
\begin{align*}
  \chi(P_i\otimes S_j)=[r]_X\left(\Sumto{l=r-j+i+1}{r-i+j-1}(X^l+X^{-l})+
    \Sumto{l=r-i-j-1}{r-j+i-1}(X^l+X^{-l})\right)=\\
  [r]_X\left((X^r+X^{-r})\Sumto{h=q}{j-i-1}(X^h+X^{-h})+
    \Sumto{l=r-i-j-1}{r-j+i-1}(X^l+X^{-l})\right)=\\
  =(X^r+X^{-r})\Sumto{k=r-j+i}{r-1-q}\chi(P_k)+
  \Sumto{k=j-i}{i+j}\chi(P_k).
\end{align*}
Finally suppose that $j >i$ and $i+j\geq r-1$  and as before let $p\in \{0,1\}$ be the parity of $r-i+j-1$.   If $p=1$ (and as above if $p=0$ or $q=0$ modify the calculation by replacing the terms $X^0+X^{-0}$ by $X^0$, still getting the same final result):
\begin{align*}
  \chi(P_i\otimes S_j)=[r]_X\left(\Sumto{l=p}{r-i+j-1}(X^{l}+X^{-l})+
    \Sumto{l=p}{i+j+1-r}(X^l+X^{-l})\right)=\\
  [r]_X\left(\Sumto{h=r-j+i+1}{r+j-i-1}X^h+\Sumto{h=r-j+i+1}{r+j-i-1}X^{-h}+
    \Sumto{h=p}{r-j+i-1}(X^h+X^{-h})+\Sumto{l=p}{i+j+1-r}(X^l+X^{-l})\right)=\\
  =[r]_X\left((X^{r}+X^{-r})\Sumto{s=q}{j-i-1}(X^s+X^{-s})+
    \Sumto{h=p}{r-j+i-1}(X^h+X^{-h})+\Sumto{l=p}{i+j+1-r}(X^l+X^{-l})\right)=\\
  =(X^r+X^{-r})\Sumto{k=r-j+i}{r-1-q}\chi(P_k)+\Sumto{k=j-i}{r-1-p}\chi(P_k)+
  \Sumto{k=2r-2-i-j}{r-1-p}\chi(P_k).
\end{align*}

To summarize the above computations, let $p,q\in \{0,1\}$ be the parities  of  $r+j-i-1$ and of  $j-i-1$, respectively. It holds:
\begin{equation}
P_i\otimes S_j=\begin{cases} 
\Opto{k=i-j}{i+j} P_k &
{\rm if}\ \left\{
\begin{array}{l}
  i\geq j\\i+j\leq r-2
\end{array}\right.
\\
\Opto{k=i-j}{r-1-p} P_k\Opto{k=2r-2-i-j}{r-1-p}P_k & 
{\rm if}\ \left\{
\begin{array}{l}
  i\geq j\\i+j\geq r-1
\end{array}\right.
\\
\Opto{k=j-i}{i+j}P_k \Opto{k=r+i-j}{r-1-q} (\C_{r}^H\oplus \C_{-r}^H)\otimes P_k& 
{\rm if}\ \left\{
\begin{array}{l}
  i<j\\i+j\leq r-2
\end{array}\right.
\\
\Opto{k=j-i}{r-1-p}P_k \Opto{k=2r-2-i-j}{r-1-p}\hspace*{-2ex}P_k 
\Opto{k=r-j+i}{r-1-q} (\C_{r}^H\oplus \C_{-r}^H)\otimes P_k& 
{\rm if}\ \left\{
\begin{array}{l}
  i<j\\i+j\geq r-1
\end{array}\right. .
\end{cases}
\end{equation} 
This is equivalent to the statement of the proposition.
\end{proof}

Let us now remark that for each $i\in \{0,1,\ldots r-2\}$,
$$\chi(P_i)=2\chi(S_i)+(\chi(\C^H_r)+\chi(\C^H_{-r}))\chi(S_{r-2-i}).$$

This, together with Proposition \ref{prop:decoPiSj} and the fact that
the modules $P_i$ are projective allow us to compute the full tensor
decomposition of $P_i\otimes P_j$ :
\begin{cor}[The tensor decomposition of $P_i\otimes P_j$]
For each $i,j\in \{0,1,\ldots r-2\}$ we have
$$P_i\otimes P_j=\left((\C^H_{r}\oplus \C^H_{-r})\otimes \left( P_i\otimes S_{r-2-j}\right) \right)\bigoplus 2\left( P_i\otimes S_j \right)$$
and so
$$P_i\otimes P_j=\left(2\Opto{k=|i-j|}{\min(i+j,r-1)} P_{k}\right)\oplus \left(2\Opto{k=2r-2-i-j}{r-1} P_{k}\right)\oplus\left(2\Opto{k=r+i-j}{r-1} P_{k}\otimes (\C^{H}_{r}\oplus \C^{H}_{-r})\right)\oplus$$
$$ \oplus\left(\Opto{k=|i+j-r+2|}{\min(i+r-j-2,r-1)}  
P_{k}\otimes(\C^{H}_{r}\oplus \C^{H}_{-r})
\right)\oplus\left(\Opto{k=r-i+j}{r-1} P_{k}\otimes (\C^{H}_{r}\oplus
\C^{H}_{-r})\right)\oplus$$
$$\oplus\left(\Opto{k=i+j+2}{r-1}
P_{k}\otimes (\C^{H}_{2r}\oplus 2\oplus \C^{H}_{-2r})\right).$$  
Similarly $V_0\otimes
P_j=\left((\C^H_{r}\oplus \C^H_{-r})\otimes \left( V_0\otimes
    S_{r-2-j}\right) \right)\bigoplus 2\left( V_0\otimes S_j
\right), \ {\rm and\ so}$
$$ V_0\otimes
P_j= \left(\Opto{k=j+1}{r-1} (\C^H_{r}\oplus \C^H_{-r})\otimes P_{k}\right )\bigoplus \Opto{k=r-1-j}{r-1} 2P_{k}.
$$ 
\end{cor}

\begin{prop}\label{P:S@S}
  Let $i,j\in\{0..r-1\}$. If $i+j\le r-1$, then 
  $$S_i\otimes S_j=\Opto{k=|i-j|}{i+j}S_k.$$
If $i+j\ge r$ then 
  $$S_i\otimes S_j=\Opto{k=|i-j|}{2r-4-i-j}S_k\oplus\Opto{k=2r-2-i-j}{r-1}P_k.$$
  In particular, semi-simple and projective modules of $\cat$ form a
  full sub-tensor category.
\end{prop}
\begin{proof}
  The proof is by induction on $i$ using that for $j\in\{1\cdots
  r-2\}$, $S_1\otimes S_j=S_{j-1}\oplus S_{j+1}$ and that $S_{r-1}$ is
  projective.  The induction is given by using
  $$S_1\otimes S_i\otimes
  S_j=(S_{i+1}\otimes S_j)\oplus (S_{i-1}\otimes S_j).$$ To see the
  last point, remark that the tensor product of two simple modules is
  a direct sum of a semi-simple module direct sum a projective module.
  Thus, the full subcategory formed by semi-simple and projective
  modules is stable by tensor product.
\end{proof}
\section{Multiplicity modules} 
Here we summarize some known facts about multiplicity modules.  The
one dimensional $\Hom$ spaces $\Hom_\cat(\C,V_{\alpha}\otimes
V_{-\alpha})$ and $\Hom_\cat(\C,V_{\alpha}\otimes V_{\beta} \otimes
V_{\gamma})$ for $\alpha+\beta+\gamma\in\{-(r-1),-(r-3),\ldots, r-1\}$
can be equipped with nice basis.  By a nice basis of
$\Hom_\cat(\C,V_{\alpha_1}\otimes\cdots \otimes V_{\alpha_n})$ we mean
a set of basis of these spaces spaces such that
\begin{enumerate}
\item it depends analytically of the parameters $\alpha_i\in\Cp$ (here
  we identify $V_\alpha$ with $\C^r=\bigoplus_i\C.v_i$ as in Equation
  \eqref{E:BasisV}) and
\item the set of basis is globally permuted by the pivotal isomorphism
  $$ 
  \Hom_\cat(\C,V_{\alpha_1}\otimes\cdots \otimes
  V_{\alpha_n}) \stackrel \sim\longrightarrow
  \Hom_\cat(\C,V_{\alpha_2}\otimes\cdots \otimes V_{\alpha_n}\otimes
  V_{\alpha_1}).$$
\end{enumerate}
The existence of a nice basis has been checked in \cite{GP1} for $r$
odd and in \cite{CM} for any $r$ but using a different normalizations.
These basis are used in \cite{GPT2,GP1,CM,CGP1,BCGP} to produce
numerical invariant of $\Cp$-colored framed trivalent graphs embedded
in $S^3$, and numerical 6j-symbols.
 
The basis of $\Hom_\cat(\C,V_{\alpha}\otimes V_{-\alpha})$ induce
isomorphisms $w_\alpha:V_{\alpha}\to V_{-\alpha}^*$ forming what is 
called a basic data (see \cite{GPT2}).  Using these isomorphisms and the
modified trace one gets a duality
$$\Hom_\cat(\C,V_{\alpha}\otimes V_{\beta} \otimes V_{\gamma}) \otimes 
\Hom_\cat(\C,V_{-\gamma}\otimes V_{-\beta} \otimes V_{-\alpha})\to\C$$
for which the basis are dual to each other.

The version $U$ of quantum $\slt$ used in \cite{CM} is slightly
different from $\Uq$.  To differentiate these algebras, let us call
$K_U,E_U,F_U\in U$ the generators, then there is a morphism of Hopf
algebras $\Uq\to U$ given by sending
$$K,E,F \text{ to respectively }K_U^2,K_UE_U,F_UK_U^{-1}.$$
through this morphism, the module $V^a$ of \cite{CM} can be
identified with the module $V_\alpha$ where $\alpha=2a-r+1$.  Then the
nice basis are given in \cite{CM} by computing some Clebsch-Gordan
coefficients.
\\
Different nice basis were computed in \cite{GP1}.  They were computed
recursively using the morphisms $X:V_\alpha\otimes V_\beta\to
V_{\alpha+1}\otimes V_{\beta+1}$ given by
$$X:v_i\otimes v_j\mapsto q^{\beta+i-j-1}\qn{\alpha-i}v_i\otimes v_{j+1}
+q^{-1}\qn{\beta-j}v_{i+1}\otimes v_{j}.$$ 
More than analytic in the parameters $\alpha_i$, they are given by
Laurent polynomials in $q^{\alpha_i}$.  But the work of \cite{GP1}
only consider odd values of $r$.

\section{Odd roots of unity}
In this section we briefly discuss the quantum group of
Subsection 2.2 when $r\in2\N+3$ is odd and $q=\e^{\frac{2\pi\sqrt{-1}}r}$ is a $r^{th}$-root
of unity.  The reason why this case is not
treated with the other are historic, technical, and due to the belief
than topological applications won't differ from the case
$q=\e^{\frac{\pi\sqrt{-1}}r}$.

Here the simple modules are 
\begin{enumerate}
\item the dimension $r$ typical modules $\{V_\alpha:\alpha\in\Cp\}$
  where now $\Cp=(\C\setminus\frac12\Z)\cup\frac r2\Z$,
\item the dimension $1$ invertible modules $\{\C_{k\frac r2}^H:k\in\Z\}$, and
\item the simple modules of dimension less than $r$: $\{S_i\otimes\C_{k\frac
    r2}^H:0<i<r,k\in\Z\}$, where the highest weight of $S_i$ is $i$.
\end{enumerate}

One difference between the odd/even case is that Ohtsuki in \cite{Oh} does not treat the case discussed in this subsection.  In any case, when $r\in2\N+3$ the category is still pivotal with the same pivot given by
$K^{r-1}$.  The fact that the formula \eqref{eq:R} still defines a
braiding on the category $\cat$ is proven in \cite[section 5.8]{GP3}.
The computation of Ohtsuki for the associated twist has never been
completed in this case.  Still in \cite{GP3} we show that a full
subcategory of $\cat$ that contain typical modules and self-dual
modules is ribbon.

If $g\in \C/2\Z\setminus (\frac12\Z)/2\Z$ then $\cat_g$ is semi-simple
and $\cat_g\subset\Proj$.  Typical modules are projective and there
exists a unique trace on $\Proj$ up to a scalar.  Its associated
modified dimension is given by formula \eqref{eq:qd}.

A nice basis for the multiplicity modules $\Hom(\C,V_\alpha\otimes
V_\beta\otimes V_\gamma)$ is missing in the literature, and the
$6j$-symbols have  not been computed in this case (they have been
computed when $q$ is a 2 times odd root of unity in \cite{GP1} and
for any even root of unity in \cite{CM} with a different
normalization).

In \cite{GPT2,GP3,CGP1} the authors construct topological invariants of
dimension $3$ using algebraic data.  
The case treated with most attention is that of quantum $\slt$ when
$q$ is a root of unity of order $2r$ but the case we have discussed in this
section is also considered as an example all together with the quantum
groups associated to the other simple Lie algebras (also at odd root
of unity).

\end{document}